\documentclass{ectj}

\usepackage{amsfonts,amssymb,graphics,epsfig,verbatim,bm,latexsym,amsmath,url,amsbsy}

\def\dim{{d}}
\def\RR{{\rm I\kern-0.18em R}}

\def\trace{{\rm  trace}}

\def\rank{{\rm  rank}}
\def\eexp{{\rm  exp}}

\def\sp#1#2{\left\langle #1, #2 \right\rangle}

\def\ess{ {\rm } }

\newcommand{\be}{\begin{eqnarray}}
\newcommand{\ee}{\end{eqnarray}}

\newcommand{\XX}{\mathcal{X}}

\newcommand{\ba}{\begin{array}}
\newcommand{\ea}{\end{array}}
\newcommand{\bs}{\begin{align}\begin{split}\nonumber}
\newcommand{\bsnumber}{\begin{align}\begin{split}}
\newcommand{\es}{\end{split}\end{align}}

\renewcommand{\(}{\left(}
\renewcommand{\)}{\right)}
\renewcommand{\[}{\left[}
\renewcommand{\]}{\right]}
\renewcommand{\hat}{\widehat}

\renewcommand{\theequation}{\thesection.\arabic{equation}}

\newtheorem{theorem}{Theorem}

\newtheorem{proposition}{Proposition}
\newtheorem{corollary}{Corollary}
\newtheorem{lemma}{Lemma}

\newtheorem{remark}{Remark}

\renewcommand{\thesection}{\arabic{section}}
\renewcommand{\theequation}{\arabic{section}.\arabic{equation}}

\renewcommand{\theproposition}{\arabic{section}.\arabic{proposition}}

\renewcommand{\thelemma}{\arabic{section}.\arabic{lemma}}

  \received{May 2012}
  \volume{10}
\title[Posterior Inference in Curved Exponential]{Posterior Inference in Curved Exponential Families under Increasing Dimensions}

\author[Belloni and Chernozhukov]{Alexandre Belloni$^{\dagger}$ and
                        Victor Chernozhukov$^{\ddagger}$}

\address{$^{\dagger}$Duke University, Fuqua School of Business,\\
                    100 Fuqua Drive, Durham, NC 27708, USA.}
\email{abn5@duke.edu}

\address{$^{\ddagger}$Massachusetts Institute of Technology, Department of Economics,\\
                    50 Memorial Drive, Cambridge, MA 02139, USA.}
\email{vchern@mit.edu}

\def\AmSTeX{$\cal A$\kern-.1667em\lower.5ex\hbox{$\cal M$}\kern-.125em
            $\cal S$-\TeX}
\def\BibTeX{{\rm B\kern-.05em{\sc i\kern-.025em b}\kern-.08em
            T\kern-.1667em\lower.7ex\hbox{E}\kern-.125emX}}

\begin{document}

  \begin{abstract}
    This work studies the large sample properties of the
posterior-based inference in the curved exponential family under
increasing dimension.  The curved structure arises from the
imposition of various restrictions on the model, such as moment restrictions, and plays a fundamental role in econometrics and others branches of
data analysis.  We establish conditions under which the posterior
distribution is approximately normal, which in turn implies
various good properties of estimation and inference procedures
based on the posterior. In the process we also revisit and improve upon previous results for
the exponential family under increasing dimension by making use of concentration of measure. We also discuss a variety of applications to high-dimensional versions of the classical econometric models including the multinomial model
with moment restrictions, seemingly unrelated regression equations, and single structural equation models. In our analysis, both the parameter dimension and
the number of moments are increasing with the sample size.


  \end{abstract}

\section{Introduction}\label{Sec:Intro}

The main motivation for this paper is to obtain large sample
results for posterior inference in the curved exponential family
under increasing dimension. In the exponential family,
the log of a density is linear in the parameters $\theta \in \Theta$;
in the curved exponential family, the parameters $\theta$ are
restricted to lie on a curve $\eta \mapsto \theta(\eta)$
parameterized by a lower dimensional parameter $\eta \in \Psi$.
There are many classical examples of densities that fall in the
curved exponential family; see for example \cite{Ef78},
\cite{lehmann}, and \cite{bandorff}. Curved exponential densities have also been
extensively used in applications \cite{Ef78,Heckman1974,hh05,hunter2007}.  An example of the condition that creates a curved structure in an exponential
family is a moment restriction of the type:
$$
\int m(x, \nu) f(x,\theta) d x =0,
$$
that restricts $\theta$ to lie on a curve that can be
parameterized as $\{\theta(\eta), \eta \in \Psi\}$, where
component $\eta=(\nu, \beta)$ contains $\nu$ and other
parameters $\beta$ that are sufficient to parameterize all
parameters $\theta \in \Theta$ that solve the above equation for
some $\nu$. In econometric applications, often moment
restrictions represent Euler equations that result from the data
being an outcome of an optimization by rational
decision-makers; see e.g.
\cite{hansen:singleton}, \cite{chamberlain}, \cite{imbens}, \cite{CH}, and \cite{Donald-Imbens-Newey}. In the last section of the paper we discuss in more details other econometric models that fit this framework, such as multivariate linear models, seemingly unrelated regressions,
single equation structural models, as in \cite{Zellner} and \cite{zellner:text}. We also discuss
multinomial model with moment restrictions. Thus, the curved exponential framework
is a fundamental complement to the exponential framework.

Under high-dimensionality, despite of its applicability,
theoretical properties of the curved exponential family are not as
well understood as the corresponding properties of the exponential
family. We contribute to the theoretical analysis
of the posterior inference in curved exponential families under high
dimensionality. We provide sufficient conditions under which
consistency and asymptotic normality of the posterior is achieved
when both the dimension of the parameter space and the sample size
are large. Our framework only requires weak conditions on the
prior distribution, which allows for improper priors. In particular, the uninformative prior always satisfies our assumptions. We also study the convergence of moments and the rates with
which we can estimate them.  We then apply these results to a variety of models where both the
parameter dimension and the number of moments are increasing with
the sample size.

The present analysis of the posterior inference in the curved
exponential family builds upon the work of
\cite{G2000} who studied posterior inference in the exponential
family under increasing dimension. Under sufficient growth
restrictions on the dimension of the model, it was shown that the
posterior distributions concentrate in neighborhoods of the true
parameter and can be approximated by an appropriate normal
distribution. Such analysis extended in a fundamental way the
classical results of \cite{Portnoy} for maximum likelihood
methods for the exponential family with increasing dimensions.

In addition to a detailed treatment of the curved exponential
family, we also revisit the exponential family setting under increasing dimension. We
present several new results that complement the results in \cite{G2000}. First, we amend the conditions on priors to allow
for a larger set of priors, for example, improper priors; second,
we use concentration inequalities for logconcave densities to
sharpen the conditions  under which the normal approximations
apply; and third, we show that the approximation of $\alpha$-th
order moments of the posterior by the corresponding moments of the
normal density becomes exponentially difficult in the moment order
$\alpha$.

We also note that by establishing the asymptotic normality of the posterior distribution we can invoke results in \cite{BelloniChernozhukov2009} that guarantees good computational properties for MCMC methods. Moreover, new results on sampling from manifolds (see \cite{Diaconis2012}) permits the implementation of different random walk schemes that are useful for implementing inference in curved exponential families. %complement the schemes analysed in \cite{BelloniChernozhukov2009}. The results derived in this work can be used to weaken the conditions required in C.1-C.3 in \cite{BelloniChernozhukov2009} to hold in curved exponential families which should also be beneficial to the new schemes proposed in \cite{Diaconis2012}.

This work allows for increasing dimension, so it can be thought as a sieve technique. However, this paper does not formally account for the approximation errors resulting from using approximate functional forms as opposed to exact functional forms. Approximation errors can be introduced into the model and our results can also be shown to hold under more stringent conditions (approximations errors need to vanish at rates faster than the sampling errors), a sharp analysis of the impact of the approximation error can be delicate and is outside of the scope of the present paper.  An example where approximation errors are controlled is the work \cite{Bontemps2011} where a (non-parametric) Gaussian regression framework with an increasing number of regressors is studied.  We view the extension of the current (non-Gaussian) setting to non-parametric cases under sharp conditions as an important direction for future work.

The rest of the paper is organized as follows. In Section
\ref{Sec:Exp} we formally define the framework, assumptions, and
develop results for the exponential family. In Section
\ref{Sec:Curved}, the main section, we develop the results for the
curved exponential family.  In Section \ref{Sec:Applications} we apply our results on
a variety of applications. Appendices %\ref{Sec:Proof}, \ref{App:Cur}, and \ref{Sec:Tech}
collect proofs of the main results and technical lemmas.

{\bf Notation.} For $a, b \in \RR^d$, their (Euclidean) inner product is denoted
by $\sp{a}{b}$, and $\|a\| = \sqrt{\sp{a}{a}}$. The unit sphere in
$\RR^d$ is denoted by $S^{d-1}= \{ v \in \RR^d : \|v \| = 1 \}$ and the $\ell_2$-ball centered at $\bar \theta$ with radius $\varepsilon>0$ is denoted by $B_\dim(\bar\theta,\varepsilon) = \{ \theta \in \RR^\dim : \|\theta-\bar\theta \| \leq \varepsilon \}$.
For a linear operator $A$, the operator norm is denoted by $\|A\|_{op}
= \sup \{ \|Aa\| : \|a\| = 1\}$. Let $\phi_d(\cdot; \mu; V)$
denote the $d$-dimensional Gaussian density function with mean
$\mu$ and covariance matrix $V$. Throughout the paper we have a triangular array of random samples $\{ X_1^{(n)} \ \ X_2^{(n)} \cdots \ X_n^{(n)}, n\geq 1\}$. For
notational convenience we will suppress the superscript $^{(n)}$ but
it is understood that ($\theta^{(n)}, \psi^{(n)}, d^{(n)}, \Theta^{(n)})$ are changing with $n$.

\section{Exponential Family Revisited}\label{Sec:Exp}

We assume that the data $\{X_i$, $i=1,\ldots,n$\} are independent $\dim$-dimensional
vectors each drawn from a $\dim$-dimensional exponential
family whose density is defined by
\begin{equation}\label{Def:Expens} f\left(x;\theta\right) = h(x)\eexp\Big(
\sp{x}{\theta} - \psi\big(\theta\big)\Big),
\end{equation}
where $\theta \in \Theta$ an open convex set of
$\RR^{\dim}$, $\psi$ is the normalizing convex
function, and $h$ depends only $x$. Let $\theta_0 \in \Theta$ denote the
(sequence of) true parameter and let $\mu =
\psi'(\theta_0)$ and $F = \psi''(\theta_0)$ be the mean and
covariance matrix of $X_i$ (with $J=F^{1/2}$ denoting its square root, i.e., $JJ = F$). We further assume that $\theta_0$ is suitable away
from the boundary of $\Theta$, namely $B_\dim(\theta_0,\sqrt{\dim\|F^{-1}\|_{op}/n}\log n) \subset \Theta$. Throughout we assume $d\to \infty$ as $n\to\infty$.

Under this framework, the posterior density of $\theta$ given the
observed data $\left\{X_i\right\}_{i=1}^n$ is defined as
\begin{equation}\label{Def:Post}
\pi_n(\theta) = \frac{\pi(\theta) \prod_{i=1}^n f(X_i;\theta)}{\int_{\Theta} \pi(\xi) \prod_{i=1}^n f(X_i;\xi) d\xi}
 = \frac{\pi(\theta) \eexp\left( \sp{\sum_{i=1}^{n} X_i}{\theta} - n \psi(\theta)
 \right)}{\int_{\Theta} \pi(\xi) \eexp\left( \sp{\sum_{i=1}^{n} X_i}{\xi} - n \psi(\xi)
 \right)d\xi} ,
\end{equation} where $\pi(\cdot)$ denotes a prior
distribution on $\Theta$.

Our results are stated in terms of a re-centered Gaussian
distribution in the local parameter space $\mathcal{U} =
\sqrt{n}J( \Theta - \theta_0)$. The
re-centering is $\Delta_n:= \sqrt{n}J^{-1}\(
\frac{1}{n}\sum_{i=1}^nX_i - \mu \)$; it follows that $E[\Delta_n]
= 0$, and $E[ \Delta_n\Delta_n' ]=I_d$ where $I_d$ denotes the $\dim$-dimensional identity matrix. Moreover, the posterior
in the local parameter space is defined for $u \in \mathcal{U}$ as
\begin{equation}\label{Def:PostLPS} \pi^*(u) = \frac{\pi(\theta_0
+ n^{-1/2} J^{-1}u) \prod_{i=1}^n f(X_i;\theta_0 + n^{-1/2}
J^{-1}u)}{\int_\mathcal{U} \pi(\theta_0 + n^{-1/2} J^{-1}u)
\prod_{i=1}^n f(X_i;\theta_0 + n^{-1/2} J^{-1}u) du}.
\end{equation}

In the same lines of \cite{Portnoy} and
\cite{G2000}, conditions on the growth rates of the third and
fourth moments are imposed. The following
quantities play an important role in the analysis:
\begin{eqnarray}
\label{Def:B1n} B_{1n}(c) &=& \sup_{\theta, a} \left\{ |E_{\theta}\left[ \sp{a}{V}^3 \right] | : a
\in S^{\dim-1}, \|J(\theta-\theta_0)\|^2 \leq \frac{c\dim}{n}
\right\}, \\
\label{Def:B2n}
B_{2n}(c) &=& \sup_{\theta, a} \left\{ E_{\theta}\left[ \sp{a}{V}^4 \right] : a
\in S^{\dim-1}, \|J(\theta-\theta_0)\|^2 \leq \frac{c\dim}{n}
\right\},\\
\label{Def:lambda}
\lambda_n(c) &=& \frac{1}{6} \left( \sqrt{\frac{c\dim}{n}}B_{1n}(0)
+ \frac{c\dim}{n} B_{2n}(c) \right),
\end{eqnarray}
where $V$ is a random variable distributed
as $J^{-1}(U-E_\theta[U])$ and $U$ has density $f(\cdot;\theta)$
as defined in (\ref{Def:Expens}). %Moreover, a combination of (\ref{Def:B1n}) and (\ref{Def:B2n}) is key to bound deviations from normality of the posterior in a neighborhood of the true parameter:

\begin{remark}
Although we focus on the i.i.d. framework the analysis can be directly extended to the case that observations are independent but not necessarily identically distributed provided the joint likelihood can still be written in the exponential family form (\ref{Def:Expens}). In this case the normalizing function satisfies $\psi = \frac{1}{n}\sum_{i=1}^n\psi_i$ where the function $\psi_i$ is induced by the $i$ observation. Similarly, the quantities $B_{1n}$ and $B_{2n}$ are defined as the average of across $i$ of $B_{1ni}(c)$ and $B_{2ni}(c)$ which are defined as in (\ref{Def:B1n}) and (\ref{Def:B2n}) for the $i$th observation.
\end{remark}

\begin{remark}
We note that $\lambda_n(c)$ is related but different from
the quantity with the same notation defined in \cite{G2000}. We provide a technical discussion about the differences in the Appendix. For now we note that (\ref{Def:lambda}) is always smaller than its counterpart in \cite{G2000} which leads to weaker requirements. In the specific applications of interest we develop bounds on (\ref{Def:lambda}). In the case that the density (\ref{Def:Expens}) is logconcave in the data, we provide generic bounds for (\ref{Def:lambda}) in Appendix D. %\ref{Sec:ControlLambda}.
\end{remark}

Next we impose some regularity conditions on the prior $\pi$.
~\\

{\bf Assumption P($c_n$).} For the specified positive sequence  $c_n$, the prior density function $\pi$ satisfies:
$$\ess\sup_{\theta\in \Theta} \ln
[\pi(\theta)/\pi(\theta_0)] \leq O(\dim) \ \mbox{and} \  \ | \ln \pi(\theta) - \ln \pi(\theta_0) | \leq K_n(c_n) \| \theta - \theta_0 \|
$$for any $\theta$ s.t. $\|\theta - \theta_0\| \leq
\sqrt{c_n\|F^{-1}\|_{op}\dim/n}$, with $ K_n(c_n)
\sqrt{c_n\|F^{-1}\|_{op}\dim/n} =o(1)$.

~\\
\indent In what follows the sequence $c_n$ will typically remain uniformly bounded in $n$.
These conditions differ from the ones imposed in \cite{G2000}. Although the same Lipschitz
condition is assumed, we require only a relative lower bound on
the value of the prior on the true parameter instead of an
absolute bound. Thus this condition requires that the true parameter does not have an exponentially small prior value relative to other parameter values. We note that such
conditions allow for improper priors which were not allowed in
\cite{G2000}. Importantly, the uninformative prior trivially satisfies
Assumption P.

Next we state the main results of this section.

\begin{theorem}\label{Thm:Main:improv}
For any fixed value $c>0$, suppose that
$(i)$ $B_{1n}(c) \sqrt{\dim/n} = o(1)$,
$(ii)$ $\lambda_n(c) \dim = o(1)$,
$(iii)$ $\|F^{-1}\|_{op}\dim/n = o(1)$, and
$(iv)$ Assumption P($c$) hold. Then we have asymptotic normality of the posterior density
function
$$ \int_{\mathcal{U}} | \pi_n^*(u) - \phi_\dim(u;\Delta_n,I_\dim)| du = o_p(1). $$
\end{theorem}

Theorem \ref{Thm:Main:improv} establishes the asymptotic normality of the posterior density
function. It has different
assumptions on the prior relative to Theorem 3 of \cite{G2000}. However, Theorem \ref{Thm:Main:improv} does not require additional technical
assumptions used in \cite{G2000}, as discussed in Appendix
A, and the growth condition of $\dim$ with relative
to the sample size $n$ is improved by $\ln \dim$ factors.

In some applications stronger
convergence properties for the posterior distribution can be required. The
following theorem provides sufficient conditions for the
$\alpha$-moment convergence. In what follows, for sequences of $\alpha$ and $d$, let
$M_{\dim,\alpha} := (d+\alpha) \left( 1 + \frac{\alpha
\ln(d+\alpha)}{d+\alpha} \right).$
\begin{theorem}\label{Thm:Main:Alpha}
In addition to the conditions (i) and (iii) of Theorem \ref{Thm:Main:improv}, suppose that the following  hold for any fixed $\bar{c}$: $(ii')$ $\lambda_n\left(\bar{c} M_{\dim,\alpha}/\dim\right)
[\bar{c}M_{\dim,\alpha}]^{1+\alpha/2}
=o(1)$;  $(iv')$  Assumption P($\bar{c}M_{\dim,\alpha}/\dim$) and {\small $K_n\left(\bar{c}M_{\dim,\alpha}/\dim\right)\sqrt{\|F^{-1}\|_{op}\big[\bar{c}M_{\dim,\alpha}\big]^{1+\alpha}/n} =o(1)$}.  Then we have
\begin{equation}\label{Result:Alpha}
\int_{\mathcal{U}} \|u\|^\alpha | \pi_n^*(u) - \phi_\dim(u;\Delta_n,I_\dim)| du = o_p(1). %
\end{equation}
\end{theorem}

Conditions $(ii')$ and $(iv')$ are strengthening of conditions $(ii)$ and $(iv)$ of Theorem \ref{Thm:Main:improv} respectively.
We emphasize that Theorem \ref{Thm:Main:Alpha} allows for $\alpha$
and $\dim$ to grow as the sample size increases. Our conditions
highlight the polynomial trade off between $n$ and $\dim$ which contrasts with an
exponential trade off between $n$ and $\alpha$. This suggests that
the estimation of higher moments in increasing dimensions
applications could be very delicate. Conditions $(ii')$ and
$(iv')$ simplify significantly if $\alpha\ln d = o(d)$, in which case  $M_{\dim,\alpha} = d(1+o(1))$.

\begin{remark}
Suppose that we are interested in allowing $\alpha$ to grow with
the sample size as well. If $\dim$ is growing in a polynomial rate
with respect to $n$, our results do not allow for $\alpha = O(\ln
n)$. Some limitation along these lines should be expected since
there is an exponential trade off between $\alpha$ and $n$.
However, it is possible to have  $\alpha = O(\sqrt{\ln n})$.
Such slow growth conditions illustrate the potential limitations
for the practical estimation of higher order moments.
\end{remark}

\section{Curved Exponential Family}\label{Sec:Curved}
Next we consider the curved exponential family. Let $X_1, X_2,\ldots,X_n$ be i.i.d. observations from a $d$-dimensional curved exponential family with density given by
$$
f(x;\theta) = h(x)\exp\( \sp{x}{\theta(\eta)} -
\psi(\theta(\eta))\),
$$
where $\eta \in \Psi \subset \RR^{\dim_1}$, $\theta:\Psi\to \Theta$,
an open subset of $\RR^d$, and $d \to \infty$ as $n \to \infty$ as
before.

The parameter of interest is $\eta$, whose true value $\eta_0$
is suitably bounded away from the boundary of $\Psi \subset \RR^{\dim_1}$ (see Assumption A).  The true value of $\theta$ induced by $\eta_0$ is given by $\theta_0 =
\theta(\eta_0)$. The mapping $\eta \mapsto \theta(\eta)$ takes
values from $\RR^{\dim_1}$ to $\RR^\dim$ where $ \dim_1 \leq
\dim$. Moreover, assume that $\eta_0$ is the unique
solution to the system $\theta(\eta) = \theta_0$.
Thus, the parameter $\theta$ corresponds to a high-dimensional
parametrization, and $\eta$ describes
the lower-dimensional parametrization of the density.

We require the following regularity conditions on the mapping
$\theta(\cdot)$ and on the prior.

\begin{itemize}
\item[] \textbf{Assumption A}. For every fixed $\kappa$, there exists a linear
operator $G:\RR^{d_{1}}\to \RR^d$ such that uniformly
in $\gamma \in B_{\dim_1}(0,\kappa N_n) \subset \sqrt{n}(\Psi-\eta_0)$, where $N_n = \sqrt{\dim}+\{\sqrt{d_{prior}} + \sqrt{\dim_1\|(G'FG)^{-1}\|_{op}} +\sqrt{\dim_1\log (\dim_1+\|J^{-1}\|_{op}/\varepsilon_0)}\}\max\{1, \|J^{-1}\|_{op}/\varepsilon_0\}$,
where $\varepsilon_0$ is defined in Assumption B, we have
\begin{equation}\label{Curved:Rep}
\sqrt{n}\left( \theta(\eta_0 + \gamma/\sqrt{n})-\theta(\eta_0)
\right) = R_{1n} + (I+R_{2n})G\gamma, \end{equation} where
\begin{equation}\label{Curved:Cond1}
 \begin{array}{l}
 \|R_{1n}\| \|J\|_{op}\{\sqrt{d}+ \|JG\|_{op} N_n\} =o(1) \ \ \mbox{and}\\
\{\|R_{2n}\|_{op}\sqrt{d} + \|JR_{2n}G\|_{op}N_n\}\|JG\|_{op}N_n =o(1).
\end{array}\end{equation}

\item[] \textbf{Assumption B}. Uniformly in $n$, there exist a
positive constant $\varepsilon_0$ bounded away from zero, such that for every $\eta \in \Psi$
we have
\begin{equation}\label{Curved:Cond2} \| \theta(\eta) -
\theta(\eta_0) \| \geq \varepsilon_0 \| \eta -
\eta_0\|.\end{equation}

\item[] {\bf Assumption P'.} The prior density function $\pi$ satisfies:
$$\sup_{\eta \in \Psi} \ln
[\pi(\theta(\eta))/\pi(\theta(\eta_0))] \leq O(\dim_{prior}).$$
\end{itemize}

Thus the mapping $\eta \mapsto \theta(\eta)$ is allowed to be
nonlinear and discontinuous. For example, the additional condition
of $R_{1n}=0$ implies the continuity of the mapping in a
neighborhood of $\eta_0$. More generally, condition
(\ref{Curved:Cond1}) impose that the map admits an (uniform)
approximate linearization in the neighborhood of $\eta_0$.

A prior $\pi$ on $\Theta$ induces a prior over $\Psi$ as $\pi(\eta) = \pi(\theta(\eta))/\int_\Psi \pi(\theta(\tilde \eta))d\tilde \eta$. Alternatively the prior can be placed directly over $\Psi$. Assumption P' also bounds the maximum log-likelihood given by the prior to any $\eta$ different than $\eta_0$ to be of the order $\dim_{prior}$ which can grow with $n$. If Assumption P holds we have $\dim_{prior} \leq \dim$. However, if the prior is placed directly on $\Psi$ we typically have $\dim_{prior} = \dim_1$. Finally, if the prior is uninformative we trivially have $\dim_{prior} = 1$. The posterior of $\eta$
given the data is denoted by
$$
\pi_n(\eta) \propto \pi(\theta(\eta)) \cdot \prod_{i=1}^n
f(x_i;\theta(\eta)) = \pi(\theta(\eta)) \cdot \exp\( n \sp{\bar
X}{\theta(\eta)} - n\psi(\theta(\eta))\)
$$ where $\bar X = (1/n)\sum_{i=1}^n X_i$.

Under this framework, we also define the local parameter space to
describe contiguous deviations from the true parameter as
$$
\gamma = \sqrt{n}(\eta - \eta_0),  \ \ \mbox{and let} \ \ s =
(G'FG)^{-1}G'\sqrt{n}(\bar X - \mu)
$$
be a first order approximation to the normalized
maximum liklelihood/ex\-tre\-mum estimate. Under this setting, the following  relations hold
for $s$: $$ E[s]= 0, \ \ E [s s'] = (G'FG)^{-1}, \ \ \mbox{and} \ \ \|s\| =
O_p(\sqrt{\dim_1\|(G'FG)^{-1}\|_{op}}).$$ The posterior density evaluated at $\gamma \in \Gamma := \sqrt{n}(\Psi-\eta_0)$ is given by  $$ \pi^*_n(\gamma) =
\ell(\gamma)/\int_{\Gamma} \ell(\gamma) d \gamma, $$ where
%{\small \bsnumber \ell(\gamma) & = \exp\( n\sp{\bar
%X}{\theta(\eta_0 + n^{-1/2}\gamma)-\theta(\eta_0)}  -n\[
%\psi(\theta
%(\eta_0 + n^{-1/2}\gamma)) - \psi(\theta (\eta_0))\] \) \\
%&\ \ \ \ \ \ \ \ \ \ \ \ \times \pi\(\theta\(\eta_0 + n^{-1/2}\gamma\)\).
%\end{split}\end{align}}
$${\small \begin{array}{rl} \ell(\gamma) & = \exp\( n\sp{\bar
X}{\theta(\eta_0 + n^{-1/2}\gamma)-\theta(\eta_0)}  -n\[
\psi(\theta
(\eta_0 + n^{-1/2}\gamma)) - \psi(\theta (\eta_0))\] \) \\
&\ \ \ \ \ \ \ \ \ \ \ \ \times \pi\(\theta\(\eta_0 + n^{-1/2}\gamma\)\).
\end{array}}$$

In order to formally state our results we use the following additional definition
$$ a_n = \sup \{ c : \lambda_n(c) \leq 1/16 \}. $$
The sequence $a_n$ characterizes a neighborhood of size
$\sqrt{a_n\dim}$ around the true local parameter for which the posterior $\ell(\cdot)$ is bounded above by a proper Gaussian density. This is useful since this neighborhood grows and controlling Gaussian tails is typically easier. In turn, this allows for weaker conditions in the next results. %Because $\lambda_n(c)$ can be used to bound deviations between the posterior distribution from a suitable Gaussian distribution (Lemma \ref{Lemma1:old}) it follows that in a neighborhood of size $\sqrt{a_n\dim}$ we can still bound the posterior $\ell(\cdot)$ by above with a proper Gaussian density. %In turn, because of logconcavity of the posterior it allows to bound the impact of the tails.

Next we address the consistency question for the maximum
likelihood estimator associated with the curved exponential
family.

\begin{theorem}\label{Thm:CurvedConsistency}
Suppose that Assumptions A, B and P' hold. Then, provided the condition $\|(G'FG)^{-1}\|_{op}\{\dim_1+\dim_{prior}\}/d = o(a_n)$ holds, we have that the maximum likelihood estimator $\hat \eta$ satisfies
$$\| \hat \eta - \eta_0 \| = O_p\left(\sqrt{\frac{\dim_1+\dim_{prior}}{n}\|(G'FG)^{-1}\|_{op}}\right). $$
\end{theorem}

Two remarks regarding Theorem \ref{Thm:CurvedConsistency} are
worth mentioning. First, we note that the condition
 $\lambda_n(c)=o(1)$ implies $a_n\to \infty$. However,  $\lambda_n(c)=o(1)$ is stronger than the condition
$\sqrt{d/n}B_{1n}(c)=o(1)$ used for consistency obtained in \cite{G2000} for the exponential family case. Second,  the consistency
result in Theorem \ref{Thm:CurvedConsistency} can be substantially impacted by the choice of prior. If the prior used is defined over the full space, it might place an exponentially small (in the dimension $d$) weight in $\eta_0$ relative to other points. In the case $\dim_1\sim \dim$ this is not problematic but it could impact the rates if $\dim_1 = o(\dim)$. In cases a prior can be placed directly over $\Gamma$ so that $d_{prior} = O(\dim_1)$, we obtain the standard rate of convergence of $\sqrt{\dim_1/n}$.

Finally, we state the asymptotic normality result for the
curved exponential family. In what follows, let $M_n:=\{1+2\|JG\|_{op}\}^2\|(G'FG)^{-1}\|_{op} \{\dim_1+d_{prior}\}$.

\begin{theorem}\label{Thm:Cexpo}
Suppose that Assumptions A, B, and P' hold. Further suppose conditions (i)-(iii) of Theorem \ref{Thm:Main:improv} hold, and for each fixed $c>0$, P($cM_n/d$), $d_1\lambda_n(cM_n/d)=o(1)$ and $\{1+2\|JG\|_{op}\}^2N_n^2/d = o(a_n)$. Then, asymptotic
normality for the posterior density associated with the curved
exponential family holds,
$$ \int | \pi_n^*(\gamma) - \phi_{\dim_1}(\gamma; s,(G'FG)^{-1})| d\gamma=o_p(1).$$
\end{theorem}

\section{Applications to Selected Econometric Models}\label{Sec:Applications}
In this section we verify the conditions that lead to asymptotic normality in a variety of econometric problems covering both exponential and curved exponential families under increasing dimension. Most examples
are motivated by the classical work of \cite{zellner:text} on Bayesian econometrics.

\subsection{Multivariate Linear Model}\label{Sec:MLM}

In this section we consider a multivariate linear model. The response variable $y_i$ is a $d_y$-dimensional vector, the disturbances $u_i$ are normally distributed with mean zero and covariance matrix $\Sigma_0$. The covariates $z_i$ are $d_z$-dimensional and the parameter matrix of interest $\Pi_0$ is $d_z\times d_y$,
\begin{equation}\label{MLM}
y_i = z_i\Pi_0 + u_i \ \ \ i = 1,\ldots,n.
\end{equation} For notational convenience, let $Y$ and $Z$ denote the matrices whose rows are given by $y_i$ and $z_i$ respectively. Note that the dimension of the model is $d = d_y^2 + d_zd_y$.

Conditioning on the covariates $Z$, this model can be cast as an exponential family model by the following parametrization
\begin{equation}\label{Par:MLM}
\theta = \left( \begin{array}{c} \theta_1 \\ \theta_2 \end{array}\right) = \left( \begin{array}{c} -\frac{1}{2} \Sigma^{-1} \\ \Pi \Sigma^{-1} \end{array}\right), \ \ \bar{X} = \frac{1}{n}\left( \begin{array}{c} \bar{X}_{1} \\ \bar{X}_{2} \end{array}\right) = \frac{1}{n}\left( \begin{array}{c} Y'Y \\ Z'Y  \end{array}\right)
\end{equation} and using the (trace) inner product $\sp{\theta}{X} = \trace(X_1'\theta_1) + \trace(X_2'\theta_2)$, see for instance \cite{Garderen}. This parametrization leads to the normalizing function
\begin{equation}\label{Par:MLM2}
 \psi(\theta) = -\frac{1}{4n} \trace(Z\theta_2\theta_1^{-1}\theta_2'Z') - \frac{1}{2} \log \det ( -2\theta_1 ).
\end{equation}
We make the following assumptions on the design. Uniformly in $n$, the covariates satisfy $\max_{i\leq n} \|z_i\| = O(d_z^{1/2})$, the matrices  $Z'Z/n$ and $\Sigma_0$ have eigenvalues bounded away from zero and from above, and the matrix $\Pi_0$ has full rank with singular values also bounded away from zero and from above.

Under these assumptions, by Lemma \ref{lemma:MultivariateLinearModelF} in the Appendix it follows that $\|F^{-1}\|_{op} = O(1)$. %Since $u \sim N(0,\Sigma)$, where $\Sigma$ is in a neighborhood of $\Sigma_0$.
Also, Lemma \ref{MultivariateLinearModelB1nB2n} in the Appendix bounds the quantities $B_{1n}(c) = O(d_z)$ and $B_{2n}(c) = O(d_z^2)$. Therefore we have asymptotic normality by Theorem \ref{Thm:Main:improv} provided that the condition
$d(d_z\sqrt{d/n}+d_z^2d/n) = o(1)$ holds.

\subsection{Seemingly Unrelated Regression Equations}\label{Sec:SUR}

The seemingly unrelated regression model (\cite{Zellner}) considers a collection of models
\begin{equation}\label{SUR}
y_{ik} = x_{ik}'\beta_k + u_{ik}, \ \ \ k=1,\ldots,d_y, \ \ i=1,\ldots,n,
\end{equation} where the dimension of $\beta_k$ is $d_k$. Let $d_z$ denote the total number of distinct covariates. The $d_y$-dimensional vector of disturbances $u$ has zero mean and covariance $\Sigma_0$ where $c<\lambda_{min}(\Sigma_0)\leq \lambda_{max}(\Sigma_0)\leq C$ for some fixed constants $c>0$ and $C<\infty$ independent of $n$. This model can be written in the form of (\ref{MLM}) by setting $\Pi = [ \pi_1(\beta_1) ; \pi_2(\beta_2) ; \cdots ; \pi_{d_y}(\beta_{d_y}) ]$. Note that the vector $ \pi_k(\beta_k)$ has zeros for regressors that do not appear in the $k$th model.
\cite{Garderen} shows that this model is a curved exponential model provided that the matrix $\Pi$ has some zero restrictions.% (that do not exclude any covariate from all models).

Consider the assumptions of Section \ref{MLM}. In this case we have that
\begin{equation}\label{SUR:c}
\eta = \left( \begin{array}{c} \eta_1 \\ \eta_2 \end{array}\right)= \left( \begin{array}{c} \Sigma^{-1} \\ \Pi \end{array}\right), \ \ \theta(\eta) =  \left( \begin{array}{c} -\frac{1}{2}\eta_1 \\ \eta_2\eta_1 \end{array}\right)= \left( \begin{array}{c} -\frac{1}{2} \Sigma^{-1} \\ \Pi \Sigma^{-1} \end{array}\right).
\end{equation} %Assume further a bounded support for $\eta$, that is $\|\eta\|_\infty\leq M$.
We restrict the space of $\Sigma$ to consider $\lambda_{min}(\Sigma) > \lambda_{min}$ a fixed constant (note that this induces $\lambda_{max}(\Sigma^{-1}) < 1/ \lambda_{min}$ which leads to a convex region in the parameter space), and that operator norm of $\Pi$ is bounded by a constant, $\|\Pi\|_{op}\leq M$ a fixed constant.

The mapping $\theta(\cdot)$ is twice differentiable and Lemma \ref{Lemma:SURtheta} establishes that condition (\ref{Curved:Rep}) holds with $R_{2n} = 0$ and $\|R_{1n}\| \leq O(N_n^2/\sqrt{n})$. Provided $\|G\|_{op}\leq C$, this implies that the requirement of $d^3\log^3 d = O(\{d_y^6 + d_z^3d_y^3\}\log^3(d_y+d_z))=o(n)$ suffices for Assumption A to hold.

Next we verify Assumption $B$ for $\varepsilon_0 = \min\{\frac{1}{4},\lambda_{min}(\Sigma_0^{-1}) / [8(1+M)]\}$. Direct calculations provide$$
\begin{array}{rcl}
\|\theta(\eta) - \theta(\eta_0)\| & \geq & \max \{ \frac{1}{2}\| \eta_1 - \eta_{01} \|,  \|\eta_2\eta_1 - \eta_{02}\eta_{01}\| \}. \\
\end{array}
$$We can assume that $\|\eta_1 - \eta_{01}\| <  2\varepsilon_0 \|\eta - \eta_0\|$ otherwise Assumption B holds. In turn, this implies that $\|\eta_2 - \eta_{02}\| \geq \frac{1}{2}\|\eta - \eta_0\|$. In this case, since the operator norm of $\eta_2$ satisfies $\|\eta_2\|_{op}\leq M$, we have
$$
\begin{array}{rcl}
\|\theta(\eta) - \theta(\eta_0)\| & \geq & \|\eta_2\eta_1 - \eta_{02}\eta_{01}\| \\
& = &  \| \eta_2 ( \eta_1 - \eta_{01}) + ( \eta_{2} - \eta_{02})\eta_{01}\| \\ %
& \geq & \| \eta_{2} - \eta_{02} \| \lambda_{min}(\Sigma_0^{-1}) - \|\eta_2\|_{op}\| \eta_1 - \eta_{01} \| \\
& \geq & \| \eta - \eta_{0} \| \lambda_{min}(\Sigma_0^{-1})/2 - M2\varepsilon_0\|\eta - \eta_{0} \|\\
& \geq & \| \eta - \eta_{0} \| \lambda_{min}(\Sigma_0^{-1})/4 \geq \varepsilon_0 \| \eta - \eta_{0} \| \\
\end{array}
$$ which verifies Assumption B.

\subsection{Single Structural Equation Model}\label{Sec:SSEM}

Next we consider the single structural equation,
\begin{equation}\label{SSEM}
y^{(1)}_i = {y^{(2)}_i}'\beta + {z^{(1)}_i}'\gamma + v_i, \ \ i=1,\ldots,n,
\end{equation}
for which the associated reduced form system, given by the multivariate linear model in (\ref{MLM}), can be partitioned as
\begin{equation}\label{SSEM-red}
(y^{(1)}_i \ \  {y^{(2)}_i}' ) = ({z^{(1)}_i}' \ \ {z^{(2)}_i}') \left( \begin{array}{cc} \pi_{11} & \pi_{12} \\ \pi_{21} & \pi_{22}\\ \end{array}\right) + ( u^{(1)}_i  \ \  {u^{(2)}_i}' ).
\end{equation}
We assume full column rank of $Z$ and $\rank(\pi_{21} \ \ \pi_{22}) = \rank(\pi_{22}) = d_y-1$ where $d_y$ is the dimension of $(y^{(1)}_i \ \ {y^{(2)}_i}')$.  The compatibility between the models (\ref{SSEM}) and (\ref{SSEM-red}) requires that
$$ \pi_{11} = \pi_{12}\beta + \gamma, \ \ \pi_{21} = \pi_{22} \beta, \ \ \mbox{and} \ \ u^{(1)}_i = {u^{(2)}_i}'\beta + v_i. $$

The model can also be embedded in (\ref{MLM}) as follows
\begin{equation}\label{SUR:c}
\eta = \left( \begin{array}{c} \eta_1 \\ \eta_2 \\ \eta_3 \end{array}\right)= \left( \begin{array}{c} \Sigma^{-1} \\ \left( \begin{array}{c} \pi_{12} \\ \pi_{22} \end{array}\right) \\ \left( \begin{array}{c} \gamma\\ \beta \end{array}\right) \\ \end{array}\right), \ \ \theta(\eta) = \left( \begin{array}{c} -\frac{1}{2} \Sigma^{-1} \\ \left( \begin{array}{cc} \gamma + \pi_{12}\beta & \pi_{12} \\ \pi_{22}\beta & \pi_{22}\\ \end{array}\right) \Sigma^{-1} \end{array}\right).
\end{equation}

Similar arguments to those used in Section \ref{Sec:SUR} show that Assumptions A and B hold
under the standard strong instrument asymptotics.

\subsection{Multinomial Model}\label{Sec:MM}

This example of multinomial model was also analyzed in \cite{G2000}. Our goal is to weaken some of the conditions required previously using the techniques proposed here.

Let $\mathcal{X} =
\{ x^0, x^1,\ldots, x^\dim\}$ be the known finite support of a
multinomial random variable $X$ where $d$ is allowed to grow with
sample size $n$. For each $j$ denote by $p_j$ the probability of
the event $\{ X = x^j \}$ which is assumed to satisfy $\max_{0\leq j\leq \dim}
1/p_j = O(\dim)$. The parameter space is given by $\theta =
(\theta_1,\ldots,\theta_\dim)$ where $\theta_j =
\log(p_j/(1-\sum_{k=1}^{\dim}p_k))$. It follows that under the assumption on the
$p_j$'s the true value of $\theta_j$'s are bounded. The Fisher
information matrix is given by $F = P - pp'$ where $P={\rm
diag}(p)$. In this case we have $B_{1n}(c)
= O(\dim^{3/2})$ and $B_{2n}(c) = O(\dim^2)$. We refer to \cite{G2000} for detailed calculations.

The growth condition $\dim^6(\log \dim)/n \to 0$ was imposed
in \cite{G2000} to obtain the asymptotic normality results (the case of
$\alpha=0$). We weaken this growth requirement by combining
the derivation in \cite{G2000} with the analysis in Section \ref{Sec:Exp}  with an uninformative
(improper) prior. In this case we have $K_n(c) = 0$ and our
definition of $\lambda_n(c)$ remove the logarithmic factors.
As a result, Theorem \ref{Thm:Main:improv} leads to a weaker growth
condition $\dim^4/n\to 0$. For
$\alpha$-moment estimation, the conditions of Theorem \ref{Thm:Main:Alpha} are
satisfied with the condition that $\dim^{4+\alpha+\delta}/n \to 0$
for any strictly positive value of $\delta$. Recently another approach based on Le Cam's proof that is specific to discrete probability distributions allows for further improvements, see \cite{BoucheronGassiat2009}. %(Moreover,the results of Theorem 2.4 of \cite{G2000} - approximation via standard Gaussian - now follow under the weaker growth condition that $d^5/n \to 0$, replacing the previous growth condition that $\dim^6(\log \dim)/n \to 0$.)

\subsection{Multinomial Model with Moment Restrictions}\label{Sec:MMMR}

In this subsection we provide a high-level discussion of the
multinomial model with moment restrictions. Let $\XX = \{ x^0,
x^1, x^2, \ldots, x^d\}$ be the known finite support of a
multinomial random variable $X$ which was described in Section
\ref{Sec:MM}. Conditions $(i)-(iv)$ are verified as in Section \ref{Sec:MM}.
% this model falls into the exponential
%family category and was studied in detail in \cite{G2000}.

As discussed in the introduction, it is of interest to incorporate
moment restrictions into this model, see \cite{chamberlain} and \cite{imbens} for discussions. This will lead to a curved exponential model as
studied in Section \ref{Sec:Curved}.

The parameter of interest is $\eta \in \Psi \subset \RR^{\dim_1}$ a
compact set. Consider a (twice continuously differentiable)
vector-valued moment function $m:\XX \times \Psi \to \RR^M$ such
that $$E[m(X,\eta)] = 0\ \ \mbox{for a unique} \ \eta_0 \in
\Psi.$$ The log-likelihood function
associated with this model
\begin{equation}\label{Def:MMlog}
 \begin{array}{l}
 \displaystyle l(q,\eta) = \sum_{i=1}^n \sum_{j =0}^d I\{X_i = x_j\} \ln
q_j \\
 \displaystyle  \mbox{for some} \ q \ \mbox{and}\ \eta \ \mbox{such that} \
\sum_{j=0}^d q_j m(x_j,\eta) = 0, \ \sum_{j=0}^d q_j = 1, \ q\geq 0,
\end{array}
\end{equation} and $l(q,\eta)=-\infty$ if the probability distribution $q$ violates any of the
 moments conditions. The log-likelihood function (\ref{Def:MMlog}) induces the mapping $q:\Psi \to \Delta^{d-1}$ formally defined as
\begin{equation}\label{Def:ThetaMM}
 \begin{array}{rl}
 \displaystyle q(\eta) =  \arg\max_{q}&  l(q,\eta) \\
&  \displaystyle \sum_{j=0}^d q_j m(x_j,\eta) = 0, \ \sum_{j=0}^d
q_j = 1, \ q \geq 0.
\end{array}
\end{equation}

In this case, the function $\theta_j(\eta)
= \log( q_j(\eta) / q_0(\eta))$ (for $j=1,\ldots,d$) is the
mapping from $\Psi \to \Theta$ discussed in Section \ref{Sec:Curved}. Assuming that
the matrix $E\[ m(X,\eta)m(X,\eta)'\]$ is uniformly positive
definite over $\eta$, \cite{Qin-Lawless} use the
inverse function theorem to show that $\theta(\cdot)$ is a twice
continuous differentiable mapping of $\eta$ in a neighborhood of
$\eta_0$. In particular this implies that we can take $R_{2n}=0$ and $\|R_{1n}\| = O( \dim \dim_1^2 (N_n^2/n) )$.
Thus, provided $G'FG$ has eigenvalues bounded away from zero and from above, $d_{prior}\leq d$, and because $\|F^{-1}\|_{op}=O(\dim)$, Assumption A holds provided that $\|R_{1n}\| N_n = O(\dim^3\dim_1^3\log \dim /n) = o(1)$. % N_n = \dim \sqrt{\dim_1\log \dim_1}

Assumption B is satisfied if $\eta$
belongs in a compact set $\Psi$ and that the mapping $\theta(\cdot)$ is
injective (over a set that contains $\Psi$ in its interior). We
refer to \cite{Newey-McFadden} for a discussion
of primitive assumptions for identification with moment
restrictions.

\section*{Appendix A: Technical Results}\label{Sec:Tech}
\renewcommand{\theequation}{A.\arabic{equation}}
\renewcommand{\thesection}{A}
\setcounter{equation}{0}
\renewcommand{\theproposition}{A.\arabic{proposition}}
\renewcommand{\thelemma}{A.\arabic{lemma}}
\medskip

In this section we prove the technical lemmas needed to prove our
main result in the following section. Our exposition follows the
work of \cite{G2000}. For the sake of completeness we
include Proposition \ref{TaylorExp}, which can be found in
\cite{Portnoy}, and a specialized version of Lemma 1 of \cite{G2000}. All the remaining proofs use different techniques
and rely on weaker assumptions. In
particular, we no longer require the prior to be proper, no bounds
on the growth of $\det \( \psi''(\theta_0) \)$ are imposed, and
$\ln n$ and $\ln d$ do not need to be of the same order.

Using the notation in Section \ref{Sec:Exp}, let \begin{equation}\label{SetH}H(a) = \{ u \in \mathcal{U} : \|u\| \leq
a \}.\end{equation} Moreover, for $u \in \mathcal{U}$ let
\begin{equation}\label{Def:tZ} \tilde{Z}_n(u) = \exp \left(
\sp{u}{\Delta_n} - \|u\|^2/2 \right) \ \ \ \mbox{and
}\end{equation} \begin{equation}\label{Def:Z} Z_n(u) = \exp\left(
\frac{1}{\sqrt{n}}\sp{\sum_{i=1}^nX_i}{ J^{-1} u} - n\[\psi\(
\theta_0+n^{-1/2}J^{-1}u\) - \psi(\theta_0)\] \right),
\end{equation} otherwise, if $\theta_0+n^{-1/2} J^{-1}u \notin
\Theta$, let $Z_n(u) = \tilde Z_n(u) = 0$. The quantity
(\ref{Def:Z}) denotes the likelihood ratio associated with $f$ as
a function of $u$. In a parallel manner, (\ref{Def:tZ}) is
associated with a standard Gaussian density. We note that (\ref{Def:tZ}) and (\ref{Def:Z}) are logconcave functions in $u$.

We start recalling a result on the Taylor expansion of $\psi$
which is key to control deviations between $\tilde Z(u)$ and
$Z(u)$.

\begin{proposition}[\cite{Portnoy}]\label{TaylorExp}
Let $\psi'$ and $\psi''$ denote respectively the gradient and the
Hessian of $\psi$. For any $\theta, \theta_0 \in \Theta$, there
exists $\tilde \theta = \lambda \theta + (1-\lambda)\theta_0$, for
some $\lambda \in [0,1]$, such that
\begin{equation}\label{Eq:Taylor}
\begin{array}{rl}
\psi(\theta) &= \psi(\theta_0) + \sp{\psi'(\theta_0)}{\theta -
\theta_0} + \frac{1}{2}\sp{\theta-\theta_0}{\psi''(\theta_0) (\theta- \theta_0)} +  \\
& +  \frac{1}{6} E_{\theta_0}[ \sp{\theta - \theta_0}{W}^3] +  \frac{1}{24} \left\{ E_{\tilde \theta}[\sp{\theta-\theta_0}{W}^4 ] - 3 \( E_{\tilde\theta}[ \sp{\theta-\theta_0}{W}^2]\)^2  \right\} \\
\end{array}
\end{equation} where $W=U-E[U]$ with $U \sim f(\cdot;\theta)$ in $E_{\theta}\[g(W)\]$.
\end{proposition}

Based on Proposition \ref{TaylorExp} we control the pointwise
deviation between $Z_n$ and $\tilde{Z}_n$ in a neighborhood of
zero (i.e., in a neighborhood of the true parameter).
\begin{lemma}[Essentially in \cite{G2000} and \cite{Portnoy}]\label{Lemma1:old}
For all $u$ such that $\|u\| \leq \sqrt{c\dim}$, we have $$ | \ln
Z_n(u) - \ln \tilde Z_n(u) | \leq \lambda_n(c)\|u\|^2 \ \
\mbox{and} \ \ \ln Z_n(u) \leq \sp{\Delta_n}{u} - \frac{1}{2}
\|u\|^2 (1- 2\lambda_n(c)).$$
\end{lemma}
%\begin{proof}
\textbf{Proof:} Define $I := \ln \tilde Z_n(u) - \ln Z_n(u)$ so that $$I = -n^{1/2}\sp{\mu}{J^{-1}u}-\frac{1}{2}\|u\|^2+
n\left\{\psi(\theta_0 + n^{-1/2}J^{-1}u) - \psi(\theta_0)\right\}.$$ Using
Proposition \ref{TaylorExp}, for some $\tilde \theta$ in the line segment $[\theta_0, \theta_0+n^{-1/2}J^{-1}u]$ we have
 $$
\begin{array}{rcl}
|I| & \leq & n\left| \frac{1}{6} E_{\theta_0}\left[\sp{\frac{u}{n^{1/2}}}{V}^3 \right] \right| + \frac{1}{24} \left|\left\{E_{\tilde \theta}\left[ \sp{\frac{u}{n^{1/2}}}{V}^4 - 3 \left(
E_{\tilde \theta}\left[\sp{\frac{u}{n^{1/2}}}{V}^2 \right]
\right)^2
\right] \right\} \right| \\
&\leq & \frac{1}{6} \left( n^{-1/2}\|u\|^3 B_{1n}(0) +
n^{-1}\|u\|^4B_{2n}(c) \right) \leq \lambda_n(c) \|u\|^2\\
\end{array}
$$where the last inequality holds by $\|u\|\leq \sqrt{cd}$ and the definition of $\lambda_n(c)$. The second statement of the lemma follows directly from the first result. $\square$ %\qed
%\end{proof}

Next we show how to bound the integrated deviation between the
quantities in (\ref{Def:tZ}) and (\ref{Def:Z}) restricted to the
neighborhood of zero.

\begin{lemma}\label{Lemma3:improv}
For any $c > 0$ we have
$$ \left( \int \tilde Z_n (u) du \right)^{-1} \int_{\{u : \|u\|\leq \sqrt{c\dim}\}} | Z_n(u) - \tilde Z_n(u) | du \leq c\dim \lambda_n(c) e^{2c\dim\lambda_n(c)}.$$
\end{lemma}
%\begin{proof}
\textbf{Proof:} Using  $|e^x - e^y| \leq |x - y | \max \{ e^x, e^y\}$ and both relations in Lemma
\ref{Lemma1:old}, for any $\|u\| \leq \sqrt{c\dim}$, we have
$$
\begin{array}{rcl}
|Z_n(u) - \tilde Z_n(u)| & \leq & | \ln Z_n(u) - \ln \tilde Z_n(u)
| \exp\left( \sp{\Delta_n}{u} - \frac{1}{2} (1-2\lambda_n(c))\|u\|^2 \right) \\
&\leq & \lambda_n(c) \|u\|^2 \exp\left( \sp{\Delta_n}{u} - \frac{1}{2} (1-2\lambda_n(c))\|u\|^2 \right). \\
\end{array}
$$
By integrating over the set $H(\sqrt{c\dim})$ as defined in (\ref{SetH}), we obtain
 $$
\begin{array}{lll}
& & \displaystyle \int_{H(\sqrt{c\dim})} | Z_n(u) - \tilde Z_n(u) | du \\
& & \leq \displaystyle
\int_{H(\sqrt{c\dim})} \lambda_n(c) \|u\|^2 \exp\left(
\sp{\Delta_n}{u} - \frac{1}{2} (1-2\lambda_n(c))\|u\|^2 \right)du \\
& & \leq \displaystyle  c\dim\lambda_n(c) \int_{H(\sqrt{c\dim})} \exp\left(
\sp{\Delta_n}{u} - \frac{1}{2} (1-2\lambda_n(c))\|u\|^2 \right)du \\
& & \leq \displaystyle  c\dim\lambda_n(c) e^{2c\dim\lambda_n(c)}
\int_{H(\sqrt{c\dim})}
\exp\left( \sp{\Delta_n}{u} - \frac{1}{2} \|u\|^2 \right)du \\
& & \leq c\dim\lambda_n(c) e^{2c\dim\lambda_n(c)} \int \tilde Z_n(u) du.\square
\end{array}
$$%\qed

%\end{proof}

The next lemma controls the tail of $Z_n$ relatively to $\tilde
Z_n$. The proof relies on the following concentration
inequality for logconcave densities functions developed in \cite{LV01}.
\begin{lemma}[Lemma 5.16 in Lov\'asz and Vempala (2007)]
Let X be a random point drawn from a distribution with a logconcave
density function $f: \mathbb{R}^m\to \mathbb{R}_+$. If $\beta \geq 2$, then
$$P\(f(X) \leq e^{-\beta (m-1)}\max_{x}f(x) \) \leq  (e^{1-\beta}\beta)^{m-1}.$$
\end{lemma}

The lemma is stated with a given
bound on the norm of $\Delta_n$ which is allowed to grow with the
dimension. Such bound on $\Delta_n$ can be easily obtained with
probability arbitrary close to one by standard arguments as $d$ grows.

\begin{lemma}\label{Lemma4:improv}
Suppose that $\|\Delta_n\|^2 < C_1 \dim$ and $\lambda_n(c) <1/8$ for some $c > 16[ 4C_1 \vee 5]$.%16\max\{ 4C_1, \ 1/(1-2\lambda_n(c))\}$.
Then for every $k\geq 1$ we have
\begin{eqnarray*} && \int_{H^c(k\sqrt{cd})} \pi(\theta_0 + n^{-1/2}J^{-1}u)Z_n(u)du \leq
 \ess \sup_{\theta\in\Theta} \pi(\theta)  \left( e^{c\dim\lambda_n(c)} \int
\tilde Z_n(u) du \right) e^{-kc\dim/16}.
 \end{eqnarray*}
\end{lemma}
%\begin{proof}
\textbf{Proof:} By definition of $c\geq C_1$, we have $\Delta_n \in H(\sqrt{c\dim})$. For any $u \in H^c(k\sqrt{c\dim})$ define $\tilde u = \sqrt{c\dim}u/\|u\|\in H(\sqrt{c\dim})$. Since $Z_n$ is a logconcave functions we have
\begin{equation}\label{Eq:tech01}
\begin{array}{rl}
\log Z_n(u) & \leq \log Z_n(\tilde u) + \nabla[\log(Z_n(\tilde u))]'(u-\tilde u)\\
& = \log Z_n(\tilde u) + (\|u\|-\sqrt{c\dim})\nabla[\log(Z_n(\tilde u))]'u/\|u\|.
\end{array}\end{equation} Next, by Lemma \ref{Lemma1:old} with $\tilde  u \in H(\sqrt{c\dim})$
\begin{equation}\label{Eq:tech02}
\begin{array}{rl}
\log Z_n(\tilde u) & \leq \log \tilde Z_n(\tilde u) + \lambda_n(c)\|\tilde u\|^2\\
& = \Delta_n'\tilde u - \frac{1}{2}(1-2\lambda_n(c))\|\tilde u\|^2\\
&\leq \sqrt{C_1\dim}\sqrt{c\dim}- \frac{1}{2}(1-2\lambda_n(c))cd\\
& =-cd( 1/2 - \lambda_n(c) - \sqrt{C_1/c}) = -cd\varphi \end{array}\end{equation}
where $\varphi :=( 1/2 - \lambda_n(c) - \sqrt{C_1/c}) \geq (1/2) - (1/8) - (1/8) = 1/4$.

Using the logconcavity of $Z_n$, (\ref{Eq:tech02}), and Lemma \ref{Lemma1:old} again we have
$$
\begin{array}{rl}
0 = \log Z_n(0) & \leq \log Z_n(\tilde u) + \nabla[\log(Z_n(\tilde u))]'(0-\tilde u)\\
& \leq -cd\varphi - \sqrt{c\dim} \nabla[\log(Z_n(\tilde u))]'u/\|u\|
\end{array}$$
so that
 \begin{equation}\label{Eq:tech03}\nabla[\log(Z_n(\tilde u))]'u/\|u\|\leq -\sqrt{c\dim}\varphi.\end{equation}

Using (\ref{Eq:tech02}) and (\ref{Eq:tech03}) in the bound (\ref{Eq:tech01}), since $\|u\|-\sqrt{c\dim} \geq (k-1)\sqrt{c\dim}$, we have for any $u\in H^c(k\sqrt{c\dim})$
$$\log Z_n(u) \leq -c\dim \varphi - \sqrt{c\dim}\varphi(k-1)\sqrt{c\dim} \leq -c\dim k\varphi.$$

Thus we can apply Lemma 5.16 in \cite{LV01} with $\beta:=\frac{\dim}{\dim-1}ck\varphi\geq ck/4 \geq 20$ and $M_f \geq Z_n(0) = 1$,  and we obtain
\begin{equation}\label{Eq:tech04}
\begin{array}{rcl}
\int_{H^c(k\sqrt{c\dim})} Z_n(u) du & \leq & (e^{1-\beta} \beta)^{\dim-1} \int Z_n(u) du \\
&\leq_{(1)} &  (e^{1-\beta} \beta)^{\dim-1} 2\int_{H(\sqrt{c\dim})} Z_n(u) du\\
&\leq_{(2)} &  (e^{1-\beta} \beta)^{\dim-1} 2e^{c\dim\lambda_n(c)} \int_{H(\sqrt{c\dim})} \tilde Z_n(u) du\\
\end{array}
\end{equation} where (1) used that  $\int Z_n(u) du \leq 2 \int_{H(\sqrt{cd})}
Z_n(u) du$ by Lemma 5.16 in \cite{LV01} with $k=1$ and $c/4\geq 20$, and (2) used Lemma \ref{Lemma1:old} since the integration is over $H(\sqrt{c\dim})$.

To prove the statement of the lemma, we have
\begin{eqnarray*} \int_{H(k\sqrt{c\dim})^c} \pi
&& (\theta_0 + n^{-1/2} J^{-1}u) Z_n(u) du \\
 && \leq
\ess\sup_{H(k\sqrt{c\dim})^c} \pi(\theta_0 + n^{-1/2}J^{-1}u)
\int_{H(k\sqrt{c\dim})^c} Z_n(u)du
\end{eqnarray*}
and the result follows by (\ref{Eq:tech04}) and noting
$$\begin{array}{rl}
2(e^{1-\beta}\beta)^{\dim-1}e^{c\dim \lambda_n(c)} & = 2e^{\dim-1}([\dim/(\dim-1)]ck\varphi)^{\dim-1}e^{-ck\dim\varphi}e^{c\dim\lambda_n(c)}\\
& \leq \exp\left( \dim + \dim\log(2ck\varphi) + c\dim\lambda_n(c)-ck\dim \varphi\right)\\
& \leq \exp\left( -ck\dim \varphi/4\right)  \end{array}$$
since $1+\log(2x) \leq x/4$ for $x\geq 19$ and $\lambda_n(c) \leq \varphi/2$.$\square$
%\end{proof}
%\qed

We note that the value of $c$ in the previous lemma could depend
on $n$ as long as the condition is satisfied. In fact, we can have
$c$ as large as
\begin{equation}\label{Def:an}
a_n := \sup \{ c : \lambda_n(c) < 1/16 \}.
\end{equation}
The sequence $a_n$ defined in (\ref{Def:an}) characterizes a neighborhood of size $\sqrt{a_n d}$
on which the quantity $Z_n(\cdot)$ can still be bounded by a
proper Gaussian. Essentially, Lemma \ref{Lemma4:improv} bounds the contribution
outside this neighborhood.

We close this section with a technical
lemma that combines some of the previous results to be easily applied.
\begin{lemma}\label{lemma:easy}
Suppose $\|\Delta_n\|^2 \leq C_1 d$, $\sup_{\theta\in\Theta} \pi(\theta)/\pi(\theta_0)\leq \exp(w\dim)$, $c > 32[4C_1\vee 5 \vee w]$, Assumption P($c$) holds, $r_{1n}:=K_n(c)\sqrt{\|F^{-1}\|_{op}c\dim/n}=o(1)$, $r_{2n}:= \frac{1}{16} - \lambda_n(c) - w/c$, and $r_{3n}:=cd\lambda_n(c)=o(1)$. Then we have:
$$\begin{array}{l}
(1) \sup_{ \|u\| \leq \sqrt{cd}} \left| \frac{\pi(\theta_0+n^{-1/2}J^{-1}u)}{\pi(\theta_0)}-1 \right| \leq (1+o(1))r_{1n}\\
(2) \int \pi(\theta_0+n^{-1/2}J^{-1}u)Z_n(u)du = [1+O(r_{1n}+e^{-c\dim r_{2n}}+r_{3n})]\int \pi(\theta_0)\tilde Z_n(u) du,\\
\end{array}
$$
\end{lemma}
%\begin{proof}
\textbf{Proof:} To show (1) note that for $x = o(1)$, we have $\exp(x) = 1 + x(1+o(1))$. Using that $\log \pi$ is Lipschitz by Assumption P($c$), and $K_n(c)\sqrt{\|F^{-1}\|_{op}c\dim/n} = o(1)$ we have
$$\begin{array}{rl}
 \sup_{ \|u\| \leq \sqrt{cd}} \left| \frac{\pi(\theta_0+n^{-1/2}J^{-1}u)}{\pi(\theta_0)}-1 \right| & \leq \left|\exp\left( K_n(c)\sqrt{\|F^{-1}\|_{op}c\dim/n}\right) - 1\right|\\
  & = (1+o(1))K_n(c)\sqrt{\|F^{-1}\|_{op}c\dim/n}\\
  \end{array}
$$where we used that $\|J^{-1}u\|\leq \sqrt{\|F^{-1}\|_{op}c\dim}$ in the specified range.

Next we show (2). By the result (1) we have
\begin{equation}\label{TwoSplit}\begin{array}{rl}
 \int \pi(\theta_0+n^{-1/2}J^{-1}u)Z_n(u)du & = \{1+O(r_{1n})\}\int_{H(\sqrt{cd})} \pi(\theta_0) Z_n(u) du \\
 & + \int_{H(\sqrt{cd})^c} \pi(\theta_0+n^{-1/2}J^{-1}u) Z_n(u)d u \\
 \end{array} \end{equation}
To control the first term in (\ref{TwoSplit}), by Lemma \ref{Lemma3:improv} we have
$$ \int_{H(\sqrt{cd})}|Z_n(u)- \tilde Z_n(u)|du \leq c\dim \lambda_n(c)\exp(2c\dim \lambda_n(c)) \int \tilde Z_n(u)du$$
where $c\dim \lambda_n(c)\exp(2c\dim \lambda_n(c)) = O(r_{3n})$ since $r_{3n}=o(1)$.

%
%$$ \int \pi(\theta_0+n^{-1/2}J^{-1}u)Z_n(u)du = \{1+O(r_{1n})\}\int_{H(\sqrt{cd})} \pi(\theta_0) Z_n(u) du + $$
%$$\ \ \ \ \ \ \ \ \ \ \ \ \ \ + O\left( \sup_{u\in \mathcal{U}}\pi(\theta_0+n^{-1/2}J^{-1}u) \int_{H(\sqrt{cd})^c} Z_n(u)d u \right). $$
To control the second term in (\ref{TwoSplit}), by  Lemma \ref{Lemma4:improv}, under the choice of $c$ above, we have
{\small $$\begin{array}{rl}
\displaystyle \int_{H(\sqrt{cd})^c} \pi(\theta_0+n^{-1/2}J^{-1}u) Z_n(u)d u &\displaystyle  \leq  \sup_{u\in\mathcal{U}}\pi(\theta_0+n^{-1/2}J^{-1}u) e^{cd\lambda_n(c)-cd/16}\int\tilde Z_n(u)du  \\
& \leq \exp(- c\dim[ 1/16 - \lambda_n(c) - w/c]) \int \pi(\theta_0)\tilde Z_n(u)du\\
& = \exp(- c\dim r_{2n}) \int \pi(\theta_0)\tilde Z_n(u)du\\
 \end{array}$$}
where we used that $\sup_{u\in\mathcal{U}}\pi(\theta_0+n^{-1/2}J^{-1}u) \leq \pi(\theta_0) \exp(w\dim)$.$\square$ %\qed %Because $c > 32w$ and $\lambda_n(c)=o(1)$, $\exp(- c\dim[ 1/16 - \lambda_n(c) - w/c]) = o(1)$ since $d\to \infty$.% (alternatively, if $d$ is fixed we could let $c\to \infty$ slowly enough so that $\lambda_n(c)=o(1)$).
%\end{proof}

\section*{Appendix B: Proofs of Results in Section 2%Theorems \ref{Thm:Main:improv} and \ref{Thm:Main:Alpha}
}\label{Sec:Proof}
\renewcommand{\theequation}{B.\arabic{equation}}
\renewcommand{\thesection}{B}
\setcounter{equation}{0}
\renewcommand{\theproposition}{B.\arabic{proposition}}
\renewcommand{\thelemma}{B.\arabic{lemma}}
\medskip

Armed with Lemmas \ref{Lemma1:old}, \ref{Lemma3:improv},
\ref{Lemma4:improv}, and \ref{lemma:easy}, we now show asymptotic
normality and $\alpha$-moment convergence results (respectively Theorem
\ref{Thm:Main:improv} and \ref{Thm:Main:Alpha}) under the
appropriate growth conditions of the dimension of the parameter
space with respect to the sample size.

It is easy to see that Theorem \ref{Thm:Main:improv} follows from
Theorem \ref{Thm:Main:Alpha} with $\alpha = 0$, therefore its
proof is omitted.

%\begin{proof}[Theorem \ref{Thm:Main:Alpha}]
\textbf{Proof of Theorem~\ref{Thm:Main:Alpha}:}
%By definition of $M_{d,\alpha}$, we use that $\sqrt{\bar{c}M_{\dim,\alpha}} \geq 4\|\Delta_n\|$ in the analysis We have that $\|\Delta_n\| = O_p(\sqrt{\dim})$ so that $\|\Delta_n\|^2\leq C_1d$.
Recall that  $\sup_\theta \pi(\theta)/\pi(\theta_0)\leq \exp(w\dim)$ by Assumption P. Since $E[\Delta_n\Delta_n']=I_d$ we have $\|\Delta_n\|=O_p(\sqrt{d})$. We will condition on the event $\{\|\Delta_n\|^2\leq C_1d\}$ which occurs with probability at least $1-1/C_1$ by Markov inequality. We will show that conditioned on $\{\|\Delta_n\|^2\leq C_1d\}$,
$V_n := \int_{\mathcal{U}} \|u\|^\alpha | \pi_n^*(u) - \phi_\dim(u;\Delta_n,I_\dim)| du = o(1).$ Thus we have for any $\epsilon>0$ and $C_1>0$, $\lim_{n\to\infty}P(|V_n|>\epsilon) \leq 1/C_1$ which implies $|V_n|=o_p(1)$.

Let $\bar c \geq 100 \vee 80w$ be sufficiently large so that conditions in Lemma \ref{lemma:easy} are satisfied with $c\dim = \bar{c}M_{d,\alpha}$ where $c > 32[4C_1\vee 5\vee w]$. Let $r_{1n}:=K_n(\bar{c}M_{d,\alpha}/d)\sqrt{\|F^{-1}\|_{op}\bar{c}M_{d,\alpha}/n}$, $r_{2n}:= [1/16 - \lambda_n(\bar{c}M_{d,\alpha}/d) - w/\bar{c}]$ (which is bounded away from zero), and $r_{3n}:=\bar{c}M_{d,\alpha}\lambda_n(\bar{c}M_{d,\alpha}/d)$.

We will divide the integral of (\ref{Result:Alpha}) in two regions
$$\Lambda = \left\{ u \in \RR^\dim : \|u\| \leq \sqrt{\bar c
M_{\dim,\alpha}} \right\} \ \ \mbox{and} \ \ \Lambda^c.$$

Note that by construction
$$ \pi_n^*(u) = \frac{\pi(\theta_0+n^{-1/2}J^{-1}u) Z_n(u)}{\int \pi(\theta_0+n^{-1/2}J^{-1}u') Z_n(u')du'} \ \ \mbox{and} \ \ \phi(u;\Delta_n,I_d) = \frac{\tilde Z_n(u)}{\int \tilde Z_n(u')du'}.$$
To simplify the notation we let $I := \int \pi(\theta_0+n^{-1/2}J^{-1}u') Z_n(u')du'$, $\tilde I := \int \tilde Z_n(u')du'$,  $\pi_u(u):= \pi(\theta_0+n^{-1/2}J^{-1}u)$ so that $\pi_u(0)= \pi(\theta_0)$. Thus we  have
 \begin{equation}\label{Eq:MainThm2}\begin{array}{rl}
 \int\|u\|^\alpha\left| \frac{\pi_u(u)Z_n(u)}{I} - \frac{\tilde Z_n(u)}{\tilde I} \right|du
& \leq \int_\Lambda \|u\|^\alpha\frac{\pi_u(u)Z_n(u)}{I}\left| \frac{\pi_u(0)}{\pi_u(u)} - 1\right| du +\\
&+\int_\Lambda\|u\|^\alpha\left| \frac{\pi_u(0)Z_n(u)}{I} - \frac{\pi_u(0)\tilde Z_n(u)}{\pi_u(0)\tilde I} \right|du \\
&+\sup_{u}\pi_u(u)\int_{\Lambda^c}\|u\|^\alpha\frac{Z_n(u)}{I}du \\
& +\int_{\Lambda^c}\|u\|^\alpha\frac{\tilde Z_n(u)}{\tilde I} du. \\
\end{array}\end{equation}
Next we bound each of the four terms on the RHS.

To bound the first integral in the RHS of (\ref{Eq:MainThm2}),
by Lemma \ref{lemma:easy}, since $r_{1n}=o(1)$, we have
$$\sup_{ u\in \Lambda} \left| \frac{\pi(\theta_0)}{\pi(\theta_0+n^{-1/2}J^{-1}u)}-1 \right| \leq (1+o(1))r_{1n}.$$
Thus,
$$\begin{array}{rl}
  \int_\Lambda \|u\|^\alpha\frac{\pi_u(u)Z_n(u)}{I}\left| \frac{\pi_u(0)}{\pi_u(u)} - 1\right| du &\leq (1+o(1))r_{1n} \int_\Lambda \|u\|^\alpha\frac{\pi_u(u)Z_n(u)}{I}du \\
& \leq  (1+o(1))r_{1n} (\bar c M_{d,\alpha})^{\alpha/2}\int_\Lambda \frac{\pi_u(u)Z_n(u)}{I}du\\
& \leq (1+o(1))r_{1n} (\bar c M_{d,\alpha})^{\alpha/2} = o(1)\end{array}$$
where the last relation follows from $(iv')$.

To bound the second integral in (\ref{Eq:MainThm2}), note that
 \begin{equation}\label{SecondTermNew}\begin{array}{rl}
\int_\Lambda\|u\|^\alpha\left| \frac{\pi_u(0)Z_n(u)}{I} - \frac{\pi_u(0)\tilde Z_n(u)}{\pi_u(0)\tilde I} \right|du & \leq \int_\Lambda\|u\|^\alpha \frac{\pi_u(0)Z_n(u)}{I}\left| \frac{\pi_u(0)\tilde I-I}{\pi_u(0)\tilde I}\right|du +\\
& + (1/\tilde I)\int_\Lambda\|u\|^\alpha\left| Z_n(u) - \tilde Z_n(u) \right|du.\\
\end{array}\end{equation}
By Lemma \ref{lemma:easy} part (2), $I = [1+O(r_{1n}+e^{-r_{2n}\dim}+r_{3n})]\pi_u(0) \tilde I.$ So that the first term in (\ref{SecondTermNew}) satisfies
$$
\int_\Lambda\|u\|^\alpha \frac{\pi_u(0)Z_n(u)}{I}\left| \frac{\pi_u(0)\tilde I-I}{\pi_u(0)\tilde I}\right|du \leq (\bar c M_{d,\alpha})^{\alpha/2} O(r_{1n}+e^{-r_{2n}\dim}+r_{3n})\int_\Lambda \frac{\pi_u(0)Z_n(u)}{I}du  $$
where $\int_\Lambda \frac{\pi_u(0)Z_n(u)}{I}du \leq 1+o(1)$ using Lemma \ref{lemma:easy} part (1). To bound the second part of (\ref{SecondTermNew}), by Lemma \ref{Lemma3:improv} we have
$$ \begin{array}{rl}
\int_\Lambda \|u\|^\alpha \left| Z_n(u) - \tilde Z_n(u) \right|du & \leq (\bar c M_{d,\alpha})^{\alpha/2}\int_\Lambda \left| Z_n(u) - \tilde Z_n(u) \right|du \\
& \leq  (\bar c M_{d,\alpha})^{\alpha/2} r_{3n}\exp(2r_{3n}) \tilde I\end{array}$$
under the condition $(\bar c M_{d,\alpha})^{\alpha/2} r_{3n}=o(1)$. Therefore we have
{\small $$\begin{array}{rl}
\int_\Lambda\|u\|^\alpha\left| \frac{\pi_u(0)Z_n(u)}{I} - \frac{\pi_u(0)\tilde Z_n(u)}{\pi_u(0)\tilde I} \right|du  \leq (\bar c M_{d,\alpha})^{\alpha/2}O(r_{1n}+e^{-r_{2n}\bar cM_{d,\alpha}}+r_{3n}) = o(1)\\
\end{array}$$} under our conditions.

To bound the third term in the RHS of (\ref{Eq:MainThm2}), let $$\Lambda_k^c := \left\{ u :
\|u\| \in \left[ k \sqrt{\bar{c} M_{\dim,\alpha}}, (k+1)
\sqrt{\bar{c} M_{\dim,\alpha}}\right] \right\}.$$
For each $k$, using Lemma \ref{Lemma4:improv}, and subsequently Lemma \ref{lemma:easy} part (2),  we have {\small
$$\begin{array}{rl}
\int_{\Lambda_k^c} Z_n(u) du
& \leq  \exp(r_{3n}-k\bar{c}M_{\dim,\alpha}/16) \int \tilde
Z_n(u) du \\
&\leq \exp(r_{3n}-k\bar{c}M_{\dim,\alpha}/16)[1+O(r_{1n}+e^{-r_{2n}\dim}+r_{3n})] I/\pi_u(0).\end{array}$$}

Thus,
{\small
\begin{equation}\label{Eq:Last}
\begin{array}{rl}
& \displaystyle \sup_u \pi_u(u) \int_{\Lambda^c} \|u\|^\alpha
\frac{Z_n(u)}{I} du \displaystyle  \leq \sup_u \pi_u(u) \sum_{k=1}^{\infty}\left\{
(k+1)^\alpha\bar{c}^{\alpha/2}M_{\dim,\alpha}^{\alpha/2}
\int_{\Lambda_k^c}Z_n(u) du\right\}\\
& \displaystyle  \leq \sup_u \frac{\pi_u(u)}{\pi_u(0)} \bar{c}^{\alpha/2}M_{\dim,\alpha}^{\alpha/2} W_n\sum_{k=1}^{\infty}
(k+1)^\alpha\exp(-k\bar{c}M_{\dim,\alpha}/16)
\end{array}
\end{equation}}\noindent where $W_n = [1+O(r_{1n}+e^{-r_{2n}\dim}+r_{3n})]\exp(r_{3n}) = 1+o(1)$. Note also that $\sup_u \pi_u(u)/\pi_u(0) \leq \exp(wd)$.
Since $M_{\dim,\alpha} > \max\{1,\alpha\}$, by $\bar{c}\geq 160$, we have
$$ \sum_{k=1}^{\infty} (k+1)^\alpha e^{-k\bar{c}M_{\dim,\alpha}/16} = \sum_{k=1}^{\infty} e^{\alpha \log(1+k) - k\bar{c}M_{\dim,\alpha}/16} \lesssim
e^{-\bar{c}M_{\dim,\alpha}/20}. $$  Moreover, our definition of $M_{\dim,\alpha}$ also implies
that the RHS of (\ref{Eq:Last}) is bounded up to a constant by:
$$\begin{array}{rl}
 W_n \exp(wd)\bar{c}^{\alpha/2}M_{\dim,\alpha}^{\alpha/2} e^{-\bar{c}M_{\dim,\alpha}/20}
 & =
W_n\exp\left( wd + \frac{\alpha }{2}( \ln \bar{c} + \ln M_{\dim,\alpha} )
- \bar{c}M_{\dim,\alpha}/20 \right)\\
&  \leq W_n\exp\left( -
\bar{c}M_{\dim,\alpha}/40\right) = o(1)\end{array}$$ by the choice of $\bar{c}$ and because $M_{\dim,\alpha} \to \infty$.

Finally, the last integral in (\ref{Eq:MainThm2}) converges to zero by standard bounds on Gaussian densities for an appropriate choice of $\bar{c}$ and $M_{\dim,\alpha} \to \infty$.$\square$ %\qed % (note that $\bar{c}$ can be chosen independently of $\dim$ and $\alpha$).
%\end{proof}

\section*{Appendix C: Proofs of Results in Section 3 % Section \ref{Sec:Curved}
}\label{App:Cur}
\renewcommand{\theequation}{C.\arabic{equation}}
\renewcommand{\thesection}{C}
\setcounter{equation}{0}
\setcounter{lemma}{0}
\setcounter{proposition}{0}
\renewcommand{\theproposition}{C.\arabic{proposition}}
\renewcommand{\thelemma}{C.\arabic{lemma}}

\medskip

For $\gamma \in \Gamma = \sqrt{n}(\Psi-\eta_0)$ let $u_\gamma =\sqrt{n}J[\theta(\eta_0+n^{-1/2}\gamma)- \theta(\eta_0)] \in \mathcal{U}\subset \RR^d$, and we write $$\begin{array}{rl}Z_n(u_\gamma)& := Z_n\Big(n^{1/2}J\[\theta\(\eta_0 +
n^{-1/2}\gamma\)-\theta(\eta_0)\]\Big)\\
& =  \exp\left(
\frac{1}{\sqrt{n}}\sp{\sum_{i=1}^nX_i}{ J^{-1} u_\gamma} - n\[\psi\(
\theta_0+n^{-1/2}J^{-1}u_\gamma\) - \psi(\theta_0)\] \right),
\end{array}$$ for $\theta_0+n^{-1/2} J^{-1}u_\gamma \in
\Theta$, and $Z_n(u)=0$ otherwise.

%In Lemma \ref{Lemma:CurvedTail} it is required that $\mu_{G,F}^2\log^2 \dim = o(a_n)$ which can be considerably weaker than the condition used in \cite{G2000} for establishing asymptotic normality for the posterior of (regular) exponential densities, $\lambda_n(c \log \dim) \dim = o(1)$. Also, the only assumption made on $\dim_1$ in Lemma \ref{Lemma:CurvedTail} was that $\dim_1 \leq \dim$. If $\dim_1\log \dim = o(\dim)$ the proof simplifies significantly since there is no need to define region (II) in the proof. Moreover, we see the role of the prior via $\dim_{prior}\leq \dim$. %Strengthening the assumption on the prior the same proof allow for the integrand on the left hand side to be $\Gamma\setminus B(0,\bar{k}\sqrt{\dim_1})$ instead of potentially $\Gamma\setminus B(0,\bar{k}\sqrt{\dim})$ in the result.

In order to state the conditions of the next lemma let
\begin{equation}\label{Def:Regions}\begin{array}{l}
N_{I} = \sqrt{\dim_1\|(G'FG)^{-1}\|_{op}}+\sqrt{\dim_{prior}\|(G'FG)^{-1}\|_{op}}, \ \ \mbox{and}\\
N_{II} = \sqrt{\dim}+\{\sqrt{d_{prior}} + \sqrt{\dim_1\|(G'FG)^{-1}\|_{op}} \\
 \ \ \ \ \ \ \ \ +\sqrt{\dim_1\log (\dim_1+\|J^{-1}\|_{op}/\varepsilon_0)}\}\max\{1, \|J^{-1}\|_{op}/\varepsilon_0\}.
\end{array}
\end{equation}

\begin{lemma}\label{Lemma:CurvedTail} For any fixed value $c>0$, suppose that $\dim_1\lambda_n(c\{1+2\|JG\|_{op}\}^2N_{I}^2/\dim) \to 0$, $ c\{1+2\|JG\|_{op}\}^2N_{II}^2/\dim = o(a_n)$,
 Assumption P($c$), A, B and P' hold.
For fixed constants $C_1, C_2$,  the relations hold for each $n$, $\|\Delta_n\| \leq \sqrt{C_1\dim}$ and $\|(G'FG)^{1/2} s\|\leq C_2\sqrt{\dim_1}$. Then we have
 we have {\small \begin{eqnarray*}  \int_{\Gamma\setminus
B_{\dim_1}(0,\bar{k}N_{I})} \pi\(\theta(\eta_0) + n^{-1/2}J^{-1}u_\gamma\)
Z_n(u_\gamma) d\gamma  = o\(\int_{\Gamma} \pi\(\theta(\eta_0) +
n^{-1/2}J^{-1}u_\gamma\) Z_n(u_\gamma) d\gamma \)
\end{eqnarray*}}
where for a constant $\bar k$ that depends on $C_1, C_2$, and Assumption P'.
\end{lemma}
%\begin{proof}[Lemma \ref{Lemma:CurvedTail}]
\textbf{Proof:}
We will consider the following partition of $\Gamma$:
$$\begin{array}{c}
(I):= \Gamma \cap B_{\dim_1}(0,\bar{k}N_{I}), \ \ (II):= \Gamma \cap B_{\dim_1}(0,\bar{k}N_{II}) \setminus B_{\dim_1}(0,\bar{k}N_{I}), \\ \mbox{and} \ \ (III) = \Gamma \cap B_{\dim_1}(0,\bar{k}N_{II})^c \end{array}$$
where the sequences  $N_{II}$ and $N_{I}$ are defined in (\ref{Def:Regions}) and the constant $\bar{k}=\bar{k}(C_1,C_2,P')$ is set large enough independent $n$. %Region $(I)$ is defined to be the region where the linear approximation $G$ for $\theta(\cdot)-\theta(\eta_0)$ is valid in the sense of Assumption A. Region $(III)$ represents the tail of the distribution. Finally, region $(II)$ is an intermediary region for which we still have interesting guarantees for deviations from normality.

Define $c_{I} = \{1+2\|JG\|_{op}\}^2\bar{k}^2N_{I}^2/\dim$
and $c_{II} = \{1+2\|JG\|_{op}\}^2\bar{k}^2N_{II}^2/\dim$. Our conditions imply
\begin{equation}\label{LCond} \dim_1\lambda_n(c_{I}) \to 0 \ \ \mbox{and} \ \ \lambda_{n}(c_{II}) <
1/16.\end{equation}

For any $\gamma \in (III)$, we have $\|\gamma\|\geq \bar{k}N_{II}$ so that by Assumption B, $\|u_\gamma\|\geq  \varepsilon_0\|\gamma\|/\|J^{-1}\|_{op} \geq \varepsilon_0 \bar{k} N_{II}/\|J^{-1}\|_{op} \geq \bar{k} N_{II} \min\{1,\varepsilon_0/\|J^{-1}\|_{op}\}$ (we denote $\tilde N_{II}:=N_{II} \min\{1,\varepsilon_0/\|J^{-1}\|_{op}\}$).

Define $\tilde{u}_\gamma = \bar{k}\tilde N_{II}
\frac{u_\gamma}{\|u_\gamma\|} \in {\mathcal U}$ so that $\|\tilde u_\gamma\| = \bar{k}\tilde N_{II} \leq \|u_\gamma\|$ (there might not be a $\tilde \gamma$ for which $u_{\tilde \gamma} = \tilde u_\gamma$). Using Lemma
\ref{Lemma1:old} we have
$$\ln Z_n(\tilde{u}_\gamma) \leq \sp{\Delta_n}{\tilde{u}_\gamma} - \frac{1}{2}(1-2\lambda_n(c_{II})) \|\tilde{u}_\gamma\|^2.$$
Since $\log Z_n(\cdot)$ is globally concave in ${\mathcal U}$ and $\log
Z_n(0) = 0$, for any $\gamma\in (III)$ \begin{equation}\label{Eq:Tail}\begin{array}{rl}
& \log
 Z_n(u_\gamma)
\leq \frac{\|u_\gamma\|}{\|\tilde u_\gamma\|} \ln
Z_n(\tilde u_\gamma) \\
%& \leq \frac{\|u_\gamma\|}{\bar k \tilde N_{II}} \( \sp{\Delta_n}{\tilde u_\gamma} - \frac{1 - 2\lambda_n(c_{II})}{2}\|\tilde u_\gamma\|^2 \) \\
& \leq \frac{\|u_\gamma\|}{\bar k \tilde N_{II}} \( \|\Delta_n\| \|\tilde u_\gamma\| - \frac{1 -
2\lambda_n(c_{II})}{2}\|\tilde u_\gamma\|^2 \) \\
& \leq  \frac{\|u_\gamma\|}{\bar k \tilde N_{II}} \( \sqrt{C_1d}\bar k\tilde N_{II} - \frac{1}{3}\bar{k}^2 \tilde N_{II}^2 \) \\
& \leq  - \|u_\gamma\| \bar{k}\tilde N_{II} /5\end{array}\end{equation} where we use that $\|\Delta_n\| \leq \sqrt{C_1\dim}\leq \bar{k}\tilde N_{II}/10$ and $\lambda_n\(c_{II}\)\leq 1/16$.

The contribution of $(III)$ can
be bounded by {\small \begin{eqnarray*}
\int_{(III)} \pi\(\theta(\eta_0) + n^{-1/2}J^{-1}u_\gamma\)
Z_n(u_\gamma) d\gamma \leq  \displaystyle \pi(\theta_0)
\ess\sup_{\eta \in \Psi} \frac{\pi(\theta(\eta))}{\pi(\theta(\eta_0))}
\int_{(III)} \eexp\( -\frac{\bar{k}\tilde N_{II}}{5}\|
u_\gamma \|\) d\gamma, \\
\end{eqnarray*}}\noindent where $\ess\sup_{\eta \in \Psi} \pi(\theta(\eta))/\pi(\theta(\eta_0)) \leq \exp(w\dim_{prior})$ for a constant $w$ associated with Assumption P'.

Since  $\gamma \in B_{\dim_1}(0,\bar{k}N_{II})^c$ implies that $\|u_\gamma\| \geq \varepsilon_0 \|\gamma\|/\|J^{-1}\|_{op}$ by Assumption B, using Lemma 5.16 in \cite{LV01} with $\beta:= \bar{k}^2 \tilde N_{II}^2/\{5(\dim_1-1)\}\geq 100$ for $\bar{k}$ large enough, direct calculations yield

$$\begin{array}{rl}
 \int_{(III)} \eexp\( -\frac{\bar{k}\tilde N_{II}}{5}\| u_\gamma \|\)
d\gamma & \leq (5\|J^{-1}\|_{op}/\{\varepsilon_0\bar{k}\tilde N_{II}\})^{\dim_1} \Gamma(d_1) {\rm Vol}_{\dim_1-1}(S^{\dim_1-1}(0,1))\\
& \times  \exp(\dim_1-1-\bar{k}^2\tilde N_{II}^2/5 +(\dim_1-1)\ln\{\bar{k}^2\tilde N_{II}^2/[5(\dim_1-1)] \})\\
& \leq  \eexp\left(-\bar{k}^2\tilde N_{II}^2/10 + C' d_1\ln (d_1+\|J^{-1}\|_{op}/\varepsilon_0)   \right)\\
\end{array}$$ for some constant $C'$, where ${\rm Vol}_k$ denote the $k$-dimensional volume of a set, $\Gamma(\dim_1)\leq {\dim_1}^{\dim_1}$ is the gamma function, and we used $\log(A)\leq A/20$ for $A\geq 100$.
%$$ {\rm dim}(x)=n, \ D\geq 4, \ X\sim f(x) \propto \eexp(-t\|x\|^p)  \Rightarrow P(\|X\|\geq (nD/[tp])^{1/p} )\leq (1/3)\eexp(-(n/p)(D/3))$$
%$$ n = d_1, t = \bar{k}N_{I} \varepsilon_0/\{5\|J^{-1}\|_{op}\}, p=1, D = \bar{k}N_{II}\varepsilon_0\bar{k}N_{I}/\{ d_15\|J^{-1}\|_{op}\}$$
%$$\begin{array}{rl}
%& \int_{(III)} \eexp\( -\frac{\bar{k}\tilde N_{II}}{5}\| u_\gamma \|\)
%d\gamma  \\
%& \leq (5\|J^{-1}\|_{op}/\bar{k}\tilde N_{II})^{\dim_1-1} {\rm Vol}_{\dim_1-1}(S^{\dim_1-1}(0,1)) \Gamma(d_1) \exp(-\varepsilon_0\bar{k}^2\tilde N_{II}N_{II}/\{15\|J^{-1}\|_{op}\})\\
%& \leq  \eexp\left( -\varepsilon_0\bar{k}^2\tilde N_{II}N_{II}/\{15\|J^{-1}\|_{op}\} + d_1\ln
%d_1 \right)\\
%\end{array}$$ where ${\rm Vol}_k$ denote the $k$-dimensional volume of a set and $\Gamma(\cdot)$ is the gamma function so that
%${\rm Vol}_{d_1-1}(S^{d_1-1}(0,1)) = 2\pi^{d_1/2}/\Gamma(d_1/2)$, ${\rm Vol}_{d_1}(B_{d_1}(0,1))=\pi^{d_1/2}/\Gamma(d_1/2 + 1)$, and ${\rm Vol}_{d_1-1}(S^{d_1-1}(0,1)) = d_1{\rm Vol}_{d_1}(B_{d_1}(0,1))$.

Using the assumption on the prior, and the definition of $\tilde N_{II}$, we can bound the contribution
of $(III)$ by
{\small \begin{equation}\label{Cont:III}
 \pi(\theta_0) \eexp\left( -\min\{1, \ \varepsilon_0/\|J^{-1}\|_{op}\}^2\frac{\bar{k}^2N_{II}^2}{10} + w\dim_{prior}+ C'd_1\ln
(d_1+\|J^{-1}\|_{op}/\varepsilon_0) \right).
\end{equation}}

Next consider $\gamma \in (II)$ where $\gamma \in
B_{\dim_1}(0,\bar{k}N_{II})\setminus B_{\dim_1}(0,\bar{k}N_{I})$. Because $\lambda_{n}(c_{II})< 1/16$, by Lemma \ref{Lemma1:old} we have
$$ \ln Z_n(u_\gamma) \leq \sp{\Delta_n}{u_\gamma} - \frac{1}{2}\frac{7}{8} \|u_\gamma\|^2. $$

By Assumption A and $\|\Delta_n\|\leq \sqrt{C_1 d}$, for any $\gamma \in
B_{\dim_1}(0,\bar{k}N_{II})$ we have
\begin{equation}\label{Well}\sp{\Delta_n}{u_\gamma} = o(1)+|\sp{G'J\Delta_n}{\gamma}|  \ \ \ \mbox{and} \ \ \ \|u_\gamma\|^2 = o(1) + \|JG\gamma\|^2.\end{equation} Combining these relations
 {\small $$\begin{array}{rcl} \displaystyle
&& \int_{(II)} \pi\(\theta(\eta_0) + n^{-1/2}J^{-1}u_\gamma\) Z_n(u_\gamma)
d\gamma \\
&& \leq  \displaystyle \pi(\theta_0) \ess \(\sup_{\eta
\in \Psi} \frac{\pi(\theta(\eta))}{\pi(\theta(\eta_0))}\)
\int_{(II)} \eexp\( \sp{\Delta_n}{u_\gamma} - \frac{1}{2}\frac{7}{8} \|u_\gamma\|^2 \) d\gamma\\
&&\leq  \displaystyle \pi(\theta_0) \ess \(\sup_{\eta \in
\Psi} \frac{\pi(\theta(\eta))}{\pi(\theta(\eta_0))}\)
(1+o(1))\int_{(II)} \eexp\(  \sp{\Delta_n}{JG\gamma}- \frac{1}{4} \|JG\gamma\|^2  \) d\gamma\\
&&=  \displaystyle \pi(\theta_0)
(1+o(1))\exp(wd_{prior}+\|(G'FG)^{1/2} s\|^2)\int_{(II)} \eexp\(  - \frac{1}{4} \|(G'FG)^{1/2}(\gamma-2s)\|^2  \) d\gamma\\
&&\leq  \displaystyle \pi(\theta_0)
(1+o(1))\exp(wd_{prior}+\|(G'FG)^{1/2} s\|^2)\int_{B_{\dim_1}(0,\bar{k}N_{I}/2)^c} \eexp\(  - \frac{1}{4} \|(G'FG)^{1/2}\gamma\|^2  \) d\gamma\\
&& \leq \pi(\theta_0) (1+o(1))\exp( wd_{prior} +\|(G'FG)^{1/2} s\|^2 - \bar{k}^2N_{I}^2/\{12\|(G'FG)^{-1}\|_{op}\}\\ && +(\dim_1/2) \log(8\pi) - \frac{1}{2}\log\det(G'FG))
\end{array}
$$} where we used $\|2 s\|\leq \|(G'FG)^{-1/2}\|_{op}\|2(G'FG)^{1/2} s\|\leq \bar{k}N_{I}/2$,  the set inclusion $(II)-2s \subset B_{\dim_1}(0,\bar{k}N_{I}/2)^c$, and standard concentration inequalities for Gaussian densities. Further using that $\|(G'FG)^{1/2} s\|\leq C_2\sqrt{\dim_1}$,  we can bound the contribution of $(II)$ by
\begin{equation}\label{Cont:II}
\pi(\theta_0) \eexp\Big( w\dim_{prior}- \bar{k}^2N_{I}^2/\{12\|(G'FG)^{-1}\|_{op}\}+\dim_1 C_2' - \frac{1}{2}\log\det(G'FG) \Big)
\end{equation} where $C_2'= 1+C_2^2+\log(8\pi)$.

Finally, we show a lower bound on the integral over $(I)$. First
 note that for any $\gamma \in (I)$ condition (\ref{Curved:Cond1})
holds and we have $u_\gamma = J(R_{1n} + (I+R_{2n})G\gamma)$.
Therefore, $u_\gamma \in B_\dim(0, \{\|JG\|_{op} + \|JR_{2n}G\|_{op}\} \bar{k}N_{I} +
\|JR_{1n}\| ) \subset B_\dim(0, \{1 + 2 \|JG\|_{op}\} \bar{k}N_{I} )$ for $n$ sufficiently large under Assumption A. Therefore \begin{equation}\label{LastPiece}\begin{array}{rcl}
&& \int_{(I)}\pi\(\theta(\eta_0) + n^{-1/2}J^{-1}u_\gamma\) Z_n(u_\gamma)
d\gamma \\
&& \geq \pi(\theta_0)\eexp\( -
K_n(c_{I})\sqrt{\frac{ c_{I}\dim}{n}} \)
\int_{(I)} Z_n(u_\gamma)
d\gamma.\\
\end{array}
\end{equation}
Under our assumptions $\eexp\( -
K_n(c_{I})\sqrt{c_{I}d/n} \) = (1+o(1))$.
Furthermore, using (\ref{Well}) and Lemma \ref{Lemma1:old}, we have \begin{equation}\label{NextLastPiece}
\begin{array}{rcl}
\ln Z_n(J\{R_{1n} + (I+R_{2n})G\gamma\}) &=&  \sp{\Delta_n}{JR_{1n}
+J(I+R_{2n})G\gamma} - \\
& & \ \ - \frac{1+2\lambda_n(c_{I})}{2}\|JR_{1n}
+J(I+R_{2n})G\gamma\|^2 \\
&\geq & o(1) + \sp{G'J\Delta_n}{\gamma} -
\frac{1+2\lambda_n(c_{I})}{2}\|JG\gamma\|^2.
\end{array}
\end{equation}
Therefore, from (\ref{LastPiece}) and (\ref{NextLastPiece}) we have {\small $$\begin{array}{lll}
& \int_{(I)}\pi\(\theta(\eta_0) + n^{-1/2}J^{-1}u_\gamma\) Z_n(u_\gamma)
d\gamma \\
& \geq  \pi(\theta_0)(1+o(1))\int_{(I)}
\eexp\(\sp{GJ\Delta_n}{\gamma} -
\frac{1+2\lambda_n(c_{I})}{2}\|JG\gamma\|^2\)d\gamma\\
&\geq  \pi(\theta_0)(1+o(1))(1 -
2\lambda_n(c_{I}))^{d_1/2} \det(G'FG)^{-1/2}\\
&\geq  \pi(\theta_0) (1+o(1)) \exp( -\frac{1}{2}\log\det(G'FG) )\\
\end{array}
$$}where we used that $\dim_1\lambda_n(c_{I}) = o(1)$, and the definition of $(I)$ with $\bar{k}$ large enough.
%since $\det(G'G)^{-1/2} \geq e^{-cd_1/2}$ and $\|G\| < c/2$ for
%some constant $c$ large enough.

The choices stated in the beginning for $N_{I}$ and $N_{II}$ yields the result provided $\bar{k}$ can is chosen sufficiently large (independent of $n$).$\square$%\qed
%\end{proof}

%\begin{proof}[Theorem \ref{Thm:CurvedConsistency}]
\textbf{Proof of Theorem~\ref{Thm:CurvedConsistency}:} Let $\hat \gamma$ be such that $\hat \eta = \eta_0 + n^{-1/2} \hat
\gamma$. By Condition P', we have $\sup_{\eta\in\Psi} |\log [\pi(\theta(\eta))/\pi(\theta(\eta_0))]| \leq w\dim_{prior}$, and it suffices to show that $\ln Z(u_{\gamma}) < -w \dim_{prior}$ for any $\gamma
\notin B_{\dim_1}(0,\bar k N_{I})$,  with  probability $1-O(1/\bar{k})$ where $N_I$ is defined as in (\ref{Def:Regions}) and $\bar k$ is sufficiently large. Indeed, in that case, since $\log Z_n(0) = 0$,
the MLE $\hat \gamma \in B_{\dim_1}(0,\bar{k} N_I)$ and the result
follows.

From (\ref{Well}) we have
$$\begin{array}{rl}
 \ln \tilde Z_n(u_\gamma) & \leq \sp{\Delta_n}{JR_{1n}+J(I+R_{2n})G\gamma}-\frac{1-2\lambda_n(c_I)}{2}\|JR_{1n}+J(I+R_{2n})G\gamma\|^2\\
&  = o(1) + \sp{G'J\Delta_n}{\gamma}-\frac{1}{2}\|JG\gamma\|^2+2\lambda_n(c_I)\|JG\gamma\|^2\\
& = o(1) + \frac{1}{2}\|(G'FG)^{1/2}s\|^2 -\frac{1}{2}\|(G'FG)^{1/2}(\gamma-s)\|^2 +2\lambda_n(c_I)\|(G'FG)^{1/2}\gamma\|^2 \\
& \leq o(1) - \frac{1-8\lambda_n(c_I)}{4}\|(G'FG)^{1/2}\gamma\|^2 \end{array} $$
provided $\|(G'FG)^{1/2}s\|\leq (1/4)\|(G'FG)^{1/2}\gamma\|$.
For $N_{I}$ as defined in (\ref{Def:Regions}) and $\bar k$ sufficiently large, with probability $1-O(1/\bar k$) we have $$\|(G'FG)^{1/2}s\|\leq C'\sqrt{d_1}\leq \bar{k}N_{I}/\{8\sqrt{\|(G'FG)^{-1}\|_{op}}\}.$$
Provided $\lambda_n(c_{I})<1/16$, which is implied by $N_I^2/d=o(a_n)$, setting $\bar{k}\geq \sqrt{16w}$, we have
$$\begin{array}{rl}
 \ln \tilde Z_n(u_\gamma) & \leq o(1) - \frac{1-8\lambda_n(c_I)}{4}\|(G'FG)^{1/2}\gamma\|^2 \\
 & \leq o(1) -\frac{1}{8}\bar{k}^2\{d_1+d_{prior}\}\end{array} $$

$\square$%\qed
%\end{proof}

%\begin{proof}[Theorem \ref{Thm:Cexpo}]
\textbf{Proof of Theorem~\ref{Thm:Cexpo}:}
Let $N_{I} = \sqrt{\dim_1\|(G'FG)^{-1}\|_{op}}+\sqrt{\dim_{prior}\|(G'FG)^{-1}\|_{op}}$.
We have that
{\small \begin{equation}\label{MainStepEq}\begin{array}{rl}
 \int | \pi_n^*(\gamma) - \phi_{d_1}(\gamma;s,(G'FG)^{-1})|d\gamma  & \leq  \int_{B_{\dim_1}(0,\bar{k}N_{I})} | \pi_n^*(\gamma) - \phi_{d_1}(\gamma;s,(G'FG)^{-1})|d\gamma   \\
&  + \int_{\Gamma\setminus B_{\dim_1}(0,\bar{k}N_{I})}  \pi_n^*(\gamma) d\gamma \\
&+ \int_{\Gamma\setminus B_{\dim_1}(0,\bar{k}N_{I})} \phi_{d_1}(\gamma;s,(G'FG)^{-1}) d\gamma. \end{array}\end{equation}}
The main step of the proof is to show that the second term in (\ref{MainStepEq}) is negligible for $\bar k$ large enough. Indeed, for $\bar k$ sufficiently large, with probability $1-O(1/C)$ we have $\|\Delta_n\|\leq  \sqrt{C d }$, and $\|(G'FG)^{1/2}s\|\leq \sqrt{C \dim_1}$ by Chebyshev inequality. Thus the conditions in Lemma \ref{Lemma:CurvedTail} hold with probability $1-O(1/\bar k)$ which implies $\int_{\Gamma\setminus B_{\dim_1}(0,\bar{k}N_{I})}  \pi_n^*(\gamma) \dim\gamma =o_p(1)$.

The last term in (\ref{MainStepEq}) is $o_p(1)$ by $\|(G'FG)^{1/2}s\|\leq C \sqrt{ \dim_1 }$,  occurring with probability $1-O(1/C)$, setting $\bar{k}$ sufficiently large, and known results for Gaussian densities.

The remaining of the proof is restricted to $B_{\dim_1}(0,\bar{k}N_{I})$. It follows the same steps in the proof of Theorem \ref{Thm:Main:Alpha} since Assumption A ensures that the linearization in (\ref{Curved:Rep}) is sufficiently precise in the region $B_{\dim_1}(0,\bar{k}N_{I})$ under (\ref{Curved:Cond1}) since $N_{I}\ll N_n$.$\square$%\qed
%\end{proof}

\section*{Appendix D: Bound on \lowercase{$\lambda_n(c)$}  when  $X$ is logconcave}\label{Sec:ControlLambda}
\renewcommand{\theequation}{D.\arabic{equation}}
\renewcommand{\thesection}{D}
\setcounter{equation}{0}
\setcounter{proposition}{0}
\setcounter{lemma}{0}
\renewcommand{\theproposition}{D.\arabic{proposition}}
\renewcommand{\thelemma}{D.\arabic{lemma}}
\medskip

In this section we derive a new bound on the fundamental quantity
$$ \lambda_n(c) = \frac{1}{6}\left( \sqrt{\frac{cd}{n}}B_{1n}(0) + \frac{cd}{n}B_{2n}(c)\right)$$
when the density function (\ref{Def:Expens}) is logconcave in the data. We start by restating the following theorem for logconcave distributions.

\begin{lemma}[Essentially in \cite{LV01}]\label{lemma:LV-R}
If $X$ is a random vector from a logconcave distribution in $\RR^d$ then
$$ E\left[ \|X\|^k \right]^{1/k} \leq 2 k E\left[ \| X\| \right] \leq 2k  E\left[ \| X\|^2 \right]^{1/2}.$$
\end{lemma}

This result provides a reverse direction of the H\"older inequality which will allow us to control higher moments based on the second moment. Since we will be bounding moments from random variables in the exponential family we can apply Lemma \ref{lemma:LV-R}.

In what follows we consider $\theta \in \mathcal{R}_c = \{\theta \in \Theta : \|J(\theta-\theta_0)\|\leq \sqrt{c\dim/n}\}$, $U \sim f_\theta = f(\cdot;\theta)$, and let $H_\theta=E_\theta[(U- E_\theta[U])(U-E_\theta[U])']^{1/2}$. In this notation $J = H_{\theta_0}$.

We first bound the third moment term $B_{1n}(0)$. In this case, since the variable of interest $\sp{a}{V}$ is properly normalized to have unit variance, its third moment is bounded by a constant.

\begin{lemma}[Bound on $B_{1n}$]\label{Lemma:B1n}
Suppose that $f(\cdot;\theta_0)$ is a logconcave distribution. Then we have that $B_{1n}(0) \leq 6^3$.
\end{lemma}
%\begin{proof}
\textbf{Proof:} Let $V = J^{-1}(U - E[U])$ where $U \sim f_{\theta_0}$. Therefore $V$ has a logconcave density function, $E[V] = 0$, and $E[VV']=I_d$.
Using Lemma \ref{lemma:LV-R}, we have
$$ B_{1n}(0) \leq \sup_{\|a\|=1} E_{\theta_0}\left[ |\sp{a}{V}|^3 \right] \leq 6^3 \sup_{\|a\|=1} E\left[ |\sp{a}{V} |^2 \right]^{3/2} = 6^3.\square %\qed
$$
%\end{proof}

Before we proceed to bound the term $B_{2n}$ in $\lambda_n$ we state and prove the following technical lemma.

\begin{lemma}\label{Lemma:Tech:M}
Let $X$ be a random vector in $\RR^d$ and $M$ be a $d\times d$ matrix. We have that
$$ \sup_{\|a\|=1} E\left[  |\sp{a}{MX}|^k \right] \leq \|M\|_{op}^k \sup_{\|a\|=1} E\left[ |\sp{a}{X}|^k\right]$$
\end{lemma}
%\begin{proof}
\textbf{Proof:} Let $\bar a$ achieve the supremum on the left hand side. Then we have
$$\begin{array}{rcl}
 E\left[  |\sp{\bar a}{MX}|^k \right]  & = & E\left[  |\sp{M'\bar a}{X}|^k \right] = \|M'\bar a\|^kE\left[  |\sp{\frac{M'\bar a}{\|M' \bar a\|}}{X}|^k \right] \\
 & \leq & \|M'\|_{op}^k\|\bar a\|^k E\left[  |\sp{\frac{M'\bar a}{\|M'\bar a\|}}{X}|^k \right] \\
 & \leq & \|M\|_{op}^k \sup_{\|a\|=1}E\left[  |\sp{a}{X}|^k \right] \\
\end{array}
$$ since $\|\bar a \| = 1$ and $\| \frac{M'\bar a}{\|M'\bar a\|}\| = 1$.$\square$%\qed
%\end{proof}

Unlike Lemma \ref{Lemma:B1n}, we need to bound the forth moment in a vanishing neighborhood of $\theta_0$. This will require an additional assumption that $H_{\theta}$ becomes sufficiently close to $ J$ for any $\theta$ in this neighborhood of $\theta_0$. %This additional condition is c than the conditions of Theorem 2.4 in \cite{G2000}.

\begin{lemma}\label{Lemma:B2n}
Suppose that $f(\cdot;\theta)$ is a logconcave distribution  and assume that $\|I - H_\theta^{-1}J \|_{op} < 1/2$. Then we have that $$\sup_{\|a\|=1} E_{\theta}\left[ |\sp{a}{V}|^k \right] \leq 2^{2k} \cdot k^{k}.$$
\end{lemma}
%\begin{proof}
\textbf{Proof:} By convexity of $t \mapsto t^k$ ($k\geq 1$) we have $(t+s)^k \leq 2^{k-1}\(t^k + s^k\)$, and Lemma \ref{Lemma:Tech:M} yields
 $$
\begin{array}{rcl}
\sup_{\|a\|=1} E_{\theta}\left[ |\sp{a}{V}|^k \right] &= & \sup_{\|a\|=1} E_{\theta}\left[ |\sp{a}{( I -  H^{-1}_\theta J +  H^{-1}_\theta J  ) V}|^k\right] \\
& \leq & 2^{k-1} \sup_{\|a\|=1} E_{\theta}\left[ |\sp{a}{( I -  H^{-1}_\theta J) V}|^k \right]+\\
&+& 2^{k-1}\sup_{\|a\|=1}E_{\theta}\left[ |\sp{a}{ H^{-1}_\theta J  V}|^k \right]\\
& \leq & 2^{k-1} \|I-H^{-1}J\|^k\sup_{\|a\|=1}E_\theta \left[ |\sp{a}{V}|^k \right] + \\
&+& 2^{k-1} \sup_{\|a\|=1} E_{\theta}\left[ |\sp{a}{ H^{-1}_\theta J  V}|^k \right].\\
\end{array}
$$

Using that $\|I-H^{-1}_\theta J\|_{op}<1/2$ we have

$$
\sup_{\|a\|=1} E_{\theta}\left[ |\sp{a}{V}|^k \right] \leq 2^k \sup_{\|a\|=1} E_{\theta}\left[ |\sp{a}{ H^{-1}_\theta J  V}|^k \right].$$

Now we invoke Lemma \ref{lemma:LV-R} to obtain

$$
\sup_{\|a\|=1} E_{\theta}\left[ |\sp{a}{V}|^k \right] \leq 2^k \cdot (2k)^k \sup_{\|a\|=1} E_{\theta}\left[ |\sp{a}{ H^{-1}_\theta J  V}|^2 \right]^{k/2} = 2^{2k} \cdot k^k
$$
since $E_\theta\[ ( H^{-1}_\theta J  V )( H^{-1}_\theta J  V )'\] = I$.$\square$%\qed
%\end{proof}

\begin{corollary}[Bound on $B_{2n}(c)$]
Suppose that $f(\cdot;\theta)$ is a logconcave distribution  and assume that $\|I - H_\theta^{-1}J \|_{op} < 1/2$ for any $\theta \in \mathcal{R}_c$. Then we have that $B_{2n}(c) \leq 2^{16}.$
\end{corollary}

\section*{Appendix E: Auxiliary results for Section 4}% \ref{Sec:Applications} }
\renewcommand{\theequation}{E.\arabic{equation}}
\renewcommand{\thesection}{E}
\setcounter{equation}{0}
\setcounter{lemma}{0}
\setcounter{proposition}{0}
\renewcommand{\theproposition}{E.\arabic{proposition}}
\renewcommand{\thelemma}{E.\arabic{lemma}}
\medskip

\begin{lemma}\label{lemma:MultivariateLinearModelF}
In the multivariate linear model, the information matrix satisfies $$\lambda_{min}(F)\geq \frac{1}{4} \frac{\lambda_{min}^4(\Sigma_0)\lambda_{min}^2(Z'Z/n)}{1+1 \vee [\lambda_{max}^2(\Sigma_0)\lambda_{max}^2(4\Pi_0'Z'Z/n)]}.$$
\end{lemma}
%\begin{proof}
\textbf{Proof:} For a direction $\gamma = (\gamma_1,\gamma_2)$, we have
$F[\gamma,\gamma] = \nabla^2 \psi(\theta_0)[\gamma,\gamma] =$
$\nabla(\nabla \psi(\theta_0)[\gamma])[\gamma]$.
Since  $$\psi(\theta) = -\frac{1}{4n} \trace(Z\theta_2\theta_1^{-1}\theta_2'Z') - \frac{1}{2} \log \det ( -2\theta_1 )$$
by direct calculations we have
$$\begin{array}{rc}\nabla \psi(\theta)[\gamma]& = (1/[4n])\trace( \gamma_1'\theta_1^{-1}\theta_2'Z'Z\theta_2 \theta_1^{-1})-(1/[2n])\trace(\gamma_2'\theta_1^{-1}\theta_2'Z'Z)+\\
& + (1/2)\trace(\gamma_1'\theta_1^{-1})\end{array}$$ and
\begin{equation}\label{BigF}\begin{array}{rl}
F[\gamma,\gamma] %& = \nabla ( \nabla \psi(\theta)[\gamma] )[\gamma]\\
%&= -(1/[2n])\trace(\gamma_1'\theta_1^{-1}\theta_2'Z'Z\theta_2\theta_1^{-1}\theta_1^{-1}\gamma_1') + (1/[2n])\trace(\gamma_2'\theta_1^{-1}\gamma_1\theta_1^{-1}\theta_2'Z'Z) +\\
%& \ \ \ + (1/[2n])\trace(\gamma_1'\theta_1^{-1}\theta_2'Z'Z \gamma_2' \theta_1^{-1})-(1/[2n])\trace(\gamma_2'Z'Z\gamma_2'\theta_1^{-1})+\\
%&\ \ \ + (1/2) \trace( \gamma_1'\theta_1^{-1}\gamma_1'\theta_1^{-1})\\
&= -(1/[2n])\trace(\gamma_1'\theta_1^{-1}\theta_2'Z'Z\theta_2\theta_1^{-1}\theta_1^{-1}\gamma_1') + \\
&+ (1/n)\trace(\gamma_2'\theta_1^{-1}\gamma_1\theta_1^{-1}\theta_2'Z'Z) -\\
& \ \ \ -(1/[2n])\trace(\gamma_2'Z'Z\gamma_2'\theta_1^{-1}) + (1/2) \trace( \gamma_1'\theta_1^{-1}\gamma_1'\theta_1^{-1})\\
%&= (4/n)\trace(\gamma_1'\Sigma_0 \Sigma_0^{-1}\Pi_0'Z'Z\Pi_0\Sigma_0^{-1}\Sigma_0\gamma_1') + (2/n) \trace(\gamma_2'\Sigma_0\gamma_1'\Sigma_0\Sigma_0^{-1}\Pi_0'Z'Z)+\\
%& + 2\trace(\gamma_1'\Sigma_0\gamma_1'\Sigma_0)+\\
%&+(2/n)\trace(\gamma_1'\Sigma_0\Sigma_0^{-1}\Pi_0'Z'Z\gamma_2'\Sigma_0)+(1/n)\trace(\gamma_2'Z'Z\gamma_2'\Sigma_0)\\
%& = (4/n)\trace(\gamma_1'\Pi_0'Z'Z\Pi_0\gamma_1')+(4/n)\trace(\gamma_1'\Pi_0'Z'Z\gamma_2'\Sigma_0)+\\
%& + 2\trace(\gamma_1'\Sigma_0\gamma_1'\Sigma_0)+(1/n)\trace(\gamma_2'Z'Z\gamma_2'\Sigma_0)\\
& = (4/n)\trace(\gamma_1'\Pi_0'Z'Z\Pi_0\gamma_1')+(4/n)\trace(\gamma_1'\Pi_0'Z'Z\gamma_2'\Sigma_0)+\\
& \ \ \ +(1/n)\trace(\Sigma_0^{1/2}\gamma_2'Z'Z\gamma_2'\Sigma_0^{1/2})+ 2\trace(\Sigma_0^{1/2}\gamma_1'\Sigma_0\gamma_1'\Sigma_0^{1/2})\\\end{array}\end{equation}
where we used that $\theta_1=-(1/2)\Sigma_0^{-1}, \theta_2=\Pi_0\Sigma_0^{-1}$.
To bound $\min_\gamma F[\gamma,\gamma]/\|\gamma\|^2$ from below let $$\mu = 1 \vee \lambda_{max}(4\Pi_0'Z'Z/n)\lambda_{max}(\Sigma_0)/[\lambda_{min}(\Sigma_0)\lambda_{min}(Z'Z/n)].$$

We consider two cases. First assume $\|\gamma_2\|\geq 2\mu \|\gamma_1\|$. Only the second term in (\ref{BigF}) can be negative and we bound its magnitude by $$\begin{array}{rl}
|(4/n)\trace(\gamma_1'\Pi_0'Z'Z\gamma_2'\Sigma_0)|& =|(4/n)\trace(\Pi_0'Z'Z\gamma_2'\Sigma_0\gamma_1')|\\
& \leq \|4\Pi_0'(Z'Z/n)\gamma_2\|\|\Sigma_0\gamma_1'\|\\
& \leq \lambda_{max}(4\Pi_0'Z'Z/n)\|\gamma_2\| \lambda_{max}(\Sigma_0)\|\gamma_1\|\\
& \leq \lambda_{min}(\Sigma_0)\lambda_{min}(Z'Z/n) \mu \|\gamma_2\|\ \|\gamma_1\|.
\end{array}$$
Since $\|\gamma_2\|$ is large in this case, we use the third term in (\ref{BigF}) to control the potential negative term, namely, we have
$$\begin{array}{rl}
F[\gamma,\gamma] & \geq (1/n)\trace(\Sigma_0^{1/2}\gamma_2'Z'Z\gamma_2'\Sigma_0^{1/2})-|(4/n)\trace(\gamma_1'\Pi_0'Z'Z\gamma_2'\Sigma_0)| \\
& \geq (1/2)\lambda_{min}(\Sigma_0)\lambda_{min}(Z'Z/n) \|\gamma_2\|^2 \\
& \geq (1/2)\lambda_{min}(\Sigma_0)\lambda_{min}(Z'Z/n)(1/(1+[1/\mu^2])\|\gamma\|^2.\end{array}$$

Otherwise, we can assume that  $\|\gamma_2\|\leq 2\mu \|\gamma_1\|$. Because $-\frac{1}{4n} \trace(Z\theta_2\theta_1^{-1}\theta_2'Z')$ is a convex function in the relevant range,\footnote{Indeed, over $\{X=(X_1,X_2):X_1\succeq 0\}$, we have $$\trace(A X_2X_1^{-1}X_2'A') = \min_{M} \trace(M) : \left[\begin{array}{cc} M & X_2'A' \\ A X_2 & X_1 \end{array}\right]\succeq 0.$$} we have
\begin{equation}\label{Eq:positive}
\begin{array}{c} (4/n)\trace(\gamma_1'\Pi_0'Z'Z\Pi_0\gamma_1')+(4/n)\trace(\gamma_1'\Pi_0'Z'Z\gamma_2'\Sigma_0)+\\ + (1/n)\trace(\Sigma_0^{1/2}\gamma_2'Z'Z\gamma_2'\Sigma_0^{1/2}) \geq 0.\end{array}\end{equation}
Therefore, by (\ref{Eq:positive}) we have  $$\begin{array}{rl}F[\gamma,\gamma] & \geq 2\trace(\Sigma_0^{1/2}\gamma_1'\Sigma_0\gamma_1'\Sigma_0^{1/2}) \\
 & \geq \lambda_{min}(\Sigma_0)^2\|\gamma_1\|^2\\
 &\geq \lambda_{min}(\Sigma_0)^2\|\gamma\|^2(1/[1+4\mu^2]).\end{array}$$
where we used $\|\gamma_2\|\leq 2\mu \|\gamma_1\|$. The result follows.$\square$%\qed
%\end{proof}

%$$ (\partial/\partial X) \trace( A X^{-1}B ) = -X^{-T}A^TB^T X^{-T}$$
%$$ (\partial/\partial X) \trace( B^TX^T C X B) = C^T X B B^T + C X B B^T$$
%$$ (\partial/\partial X)  Y(X)^{-1} = - Y(X)^{-1} [(\partial/\partial X)  Y(X)]   Y(X)^{-1} $$
%$$ \sp{\gamma}{X} = \trace(X'\gamma)$$
%$$ (\partial/\partial X) \log \det ( X ) = -X^{-1}$$
%$$ \log \det( -\theta_1 + t \gamma ) = \log \det( -\theta_1 ) + t \nabla \log \det(-\theta_1)[\gamma] + ... $$
%$$ (1/2) \theta_1^{-1}[\gamma] = (1/2)\trace( \gamma^T \theta_1^{-1})$$
%$$ f(X+t\gamma) = f(X) + \nabla f(X)[t\gamma] +....  $$
%$$\nabla f(X)[t\gamma] = \trace(t\gamma'\nabla f(X))=t \trace(\gamma'\nabla f(X)) $$
%$$ \nabla^2 f(X)[\gamma, \tilde\gamma] = \nabla ( \nabla f(X)[\gamma])[\tilde \gamma]$$

\begin{lemma}\label{MultivariateLinearModelB1nB2n}
In the multivariate linear model, we have
$$ B_{1n}(c) = O(d_z) \ \ \mbox{and} \ \ B_{2n}(c) = O(d_z^2)$$ where the constants can depend on the maximal eigenvalues of $J^{-1}$ and $\Sigma_0$, and maximal singular value of $\Pi_0$.
\end{lemma}
%\begin{proof}
\textbf{Proof:} Let $y_i \in \RR^{d_y}$,
$X_i = (y_iy_i', z_iy_i' )$ we have
$$ y_iy_i' = (u_i+z_i\Pi)(u_i' + \Pi'z_i') = u_iu_i' + u_i\Pi'z_i' + z_i\Pi u_i'+z_i\Pi\Pi'z_i'$$ so that if $u_i\sim N(0,\Sigma)$, we have $ y_iy_i' - E[y_iy_i'] = u_iu_i' - \Sigma + u_i\Pi'z_i' + z_i\Pi u_i'$.
Similarly,
$$ z_iy_i' = z_i(u_i' + \Pi'z_i') = z_iu_i' + z_i\Pi'z_i'$$ so that if $u_i\sim N(0,\Sigma)$, we have $ z_iy_i' - E[z_iy_i'] = z_iu_i'$. Thus, for $a=(a_1',a_2')'$ and $J^{-1} = [ J^{-1}_1 ; J^{-1}_2]$,
$$\begin{array}{rl}
 \sp{a}{J^{-1}(X - E[X])}& = \trace(a_1'J^{-1}_1[u_iu_i' - \Sigma + u_i\Pi'z_i' + z_i\Pi u_i']) \\
 & \ \ + \trace(a_2'J^{-1}_2[z_iu_i']).\\
\end{array}
 $$
Define $W_0 = u_iu_i'-\Sigma$, $W_1=u_i\Pi'z_i' + z_i\Pi u_i'$, $W_2=z_iu_i'$, $\tilde a_1 = a_1'J^{-1}_1$ and $\tilde a_2 = a_2'J^{-1}_2$.
By the symmetry of the probability distribution of $u$, we have
 $$ \begin{array}{rcl}
&& |E\[ \sp{a}{J^{-1}(X - E[X])}^3 \]| \\
%&=& |E[ \trace^3(\tilde a_1W_0) + 3\trace(\tilde a_1W_0)\trace^2(\tilde a_1 W_1)+ 3\trace^2(\tilde a_1W_0)\trace(\tilde a_1 W_1)\\
%&& +\trace^3(\tilde a_1W_1) + 3\trace(\tilde a_1(W_0+W_1))\trace^2(\tilde a_2W_2) \\
%&& + 3\trace^2(\tilde a_1(W_0+W_1))\trace(\tilde a_2W_2)+\trace^3(\tilde a_2 W_2)]|\\
&& = |E[ \trace^3(\tilde a_1W_0) + 3\trace(\tilde a_1W_0)\{\trace^2(\tilde a_1 W_1)+\trace^2(\tilde a_2W_2)\}\\
&& + 6 \trace(\tilde a_1W_0)\trace(\tilde a_1W_1)\trace(\tilde a_2W_2)]|\\
&& \leq E[ |\trace(\tilde a_1W_0)|^3]+ 3E[ |\trace(\tilde a_1W_0)|^3]^{1/3} E[ |\trace(\tilde a_1W_1)|^3]^{2/3}\\
&& +  3E[ |\trace(\tilde a_1W_0)|^3]^{1/3}E[ |\trace(\tilde a_2W_2)|^3]^{2/3} \\
&& + 6E[ |\trace(\tilde a_1W_0)|^3]^{1/3}E[ |\trace(\tilde a_1W_1)|^3]^{1/3}E[ |\trace(\tilde a_2W_2)|^3]^{1/3}\\

%&=& |E\[\trace^3(a_1'J^{-1}_1[u_iu_i' - \Sigma + u_i\Pi'z_i' + z_i\Pi u_i'])\] +\\
%&& + 3E\[\trace^2(a_1'J^{-1}_1[u_iu_i' - \Sigma + u_i\Pi'z_i' + z_i\Pi u_i'])\trace(a_2'J^{-1}_2[z_iu_i'])\]+\\
%&  & +3 E\[ \trace(a_1'J^{-1}_1[u_iu_i' - \Sigma + u_i\Pi'z_i' + z_i\Pi u_i'])\trace^2(a_2'J^{-1}_2[z_iu_i']) \]+\\
%&& + E\[\trace^3(a_2'J^{-1}_2[z_iu_i']) \]|\\
%&& = |E\[\trace^3(a_1'J^{-1}_1[u_iu_i' - \Sigma ])\] +\\
%&  & \ \ +3 E\[ \trace(a_1'J^{-1}_1[u_iu_i' - \Sigma])\trace^2(a_2'J^{-1}_2[z_iu_i']) \]|\\
%&& \leq 2^3|E\[\trace^3(a_1'J^{-1}_1u_iu_i')\]|+2^3|\trace^3(a_1'J^{-1}_1\Sigma)| +\\
%&  & \ \ +3  E\[\trace(a_1'J^{-1}_1u_iu_i')\trace^2(a_2'J^{-1}_2z_iu_i') \]|+\\
%& & \ \ +3\trace(a_1'J^{-1}_1\Sigma)E\[\trace^2(a_2'J^{-1}_2z_iu_i') \]|\\
%& & \leq 2^3|E\[(u_i'a_1'J^{-1}_1u_i)^3\]|+2^3|\trace^3(a_1'J^{-1}_1\Sigma)| +\\
%&  & \ \ +3  E\[(u_i'a_1'J^{-1}_1u_i)(u_i'a_2'J^{-1}_2z_i)^2 \]|+\\
%& & \ \  +3\trace(a_1'J^{-1}_1\Sigma)E\[(u_i'a_2'J^{-1}_2z_i)^2\]|\\
%& \leq & E[ \|J^{-1}_1a_1\|^3\| \|u_iu_i' - \Sigma\|^3] +\\
%& & + 3E[\|J^{-1}_1a_1\|\| \|u_iu_i' - \Sigma\|\|J^{-1}_2a_2\|^2\|z_iu_i'\|^2]\\
& & = O(d_z)\end{array} $$\noindent which follows from $$E[ |\trace(\tilde a_1W_0)|^3] \leq C \ \ \mbox{and} \ \ \max\{E[ |\trace(\tilde a_1W_1)|^3], E[ |\trace(\tilde a_2W_2)|^3] \} \lesssim d_z^{3/2}.$$ The two inequalities above follow from the Gaussianity of $u$, Lemma \ref{lemma:LV-R},  $J$ has eigenvalues bounded away from zero and from above uniformly in $n$, so that $\|\tilde a_1\|\leq C$ and $\|\tilde a_2\|\leq C$, $\|z_i\| \lesssim d_z^{1/2}$, and $\Pi$ having bounded singular eigenvalues.

To obtain the second result, we have
{\small $$ \begin{array}{rcl}
E\[ \sp{a}{J^{-1}(X - E[X])}^4 \] &\leq & 2^4E\[\trace^4(a_1'J^{-1}_1[u_iu_i' - \Sigma + u_i\Pi'z_i' + z_i\Pi u_i'])\]+\\
& & + 2^4E\[\trace^4(a_2'J^{-1}_2z_iu_i') \].
\end{array}$$}
Similar calculations yield
$$ E\[ \sp{a}{J^{-1}(X - E[X])}^4 \] = O(d_z^2).\square$$%\qed$$

%\end{proof}

\begin{lemma}\label{Lemma:SURtheta} In the seemingly unrelated regressors model, for every $\kappa >0$, uniformly in $\gamma=(\gamma_1,\gamma_2)\in B_{\dim_1}(0,\kappa N_n)$ we have
$$\sqrt{n}(\theta(\eta_0+\gamma/\sqrt{n})- \theta(\eta_0)) = \left( \begin{array}{c} -\gamma_1/2 \\ \Pi_0\gamma_1 + \gamma_2\Sigma_0^{-1} \end{array}\right)+R_{1n}(\gamma) $$ where $\|R_{1n}(\gamma)\|\leq\kappa^2N_n^2/\sqrt{n}$.
\end{lemma}
%\begin{proof}
\textbf{Proof:} By direct calculations we have
$$ \nabla\theta(\eta)[\gamma] = \left( \begin{array}{c} -\gamma_1/2 \\ \Pi\gamma_1 + \gamma_2\Sigma^{-1} \end{array}\right) \ \ \mbox{and} \ \ \nabla^2\theta(\eta)[\gamma,\gamma] = \left( \begin{array}{c} 0 \\ 2\gamma_2\gamma_1 \end{array}\right).$$
It follows that $\|2\gamma_2\gamma_1\| \leq 2\|\gamma_2\| \ \|\gamma_1\| \leq \|\gamma_2\|^2 + \|\gamma_1\|^2$, so that $\|\nabla^2\theta(\eta)[\gamma,\gamma]\| \leq \|\gamma\|^2$ for all $\eta$.
Thus, for any $\gamma \in B_{\dim_1}(0,\kappa N_n)$, we can set $R_{2n}=0$ and $R_{1n}$ satisfies
$$ \|R_{1n}(\gamma)\| \leq \sqrt{n} \sup_{\eta}\|\nabla^2\theta(\eta)[\gamma/\sqrt{n},\gamma/\sqrt{n}]\| \leq \|\gamma\|^2/\sqrt{n}\leq \kappa^2N_n^2/\sqrt{n}.\square$$%\qed$$

%\end{proof}

 \section*{Acknowledgements}
The authors gratefully acknowledge the research support from an NSF grant.  The authors would like to thank the editor and the referees for the useful comments.

\end{document}